\documentclass{amsart}

\usepackage{amsmath, amssymb, amsthm, enumerate}
\theoremstyle{plain}
  \newtheorem{thm}{Theorem}
  
  \newtheorem{prop}{Proposition}

\theoremstyle{definition}
  
  \newtheorem*{rem}{Remark}

\newcommand{\g}{\mathfrak{g}}
\newcommand{\iti}{\tilde\imath}
\newcommand{\real}{{\mathbb R}}
\newcommand{\CA}{\mathcal{CA}}
\newcommand{\SA}{\mathcal{SA}}
\newcommand{\G}{\widehat{\mathcal{G}}}
\DeclareMathOperator{\End}{End}
\DeclareMathOperator{\ad}{ad}

\newcommand{\I}{\mathcal{I}}
\newcommand{\R}{\mathbb{R}}

\begin{document}

\title{Equivariant cohomology and current algebras}

\author{Anton Alekseev}
\address{Section of Mathematics, University of Geneva, 2-4 rue du Li\`evre, c.p. 64, 1211
Gen\`eve 4, Switzerland}
\email{Anton.Alekseev@unige.ch}

\author{Pavol \v Severa}
\address{Section of Mathematics, University of Geneva, 2-4 rue du Li\`evre, c.p. 64, 1211
Gen\`eve 4, Switzerland\\ on leave from FMFI UK, Bratislava, Slovakia}
\email{Pavol.Severa@gmail.com}

\begin{abstract}
This paper touches upon two big themes, equivariant cohomology and current algebras.
Our first main result is as follows: we define a pair of current algebra functor which assigns 
Lie algebras (a current algebras) $\CA(M, A)$ and $\SA(M, A)$  to a manifold $M$ and a differential
graded Lie algebra (DGLA) $A$. The functors $\CA$ and $\SA$ are contravariant with respect to $M$ 
and covariant with respect to $A$. If $A=C\g$, the cone of a Lie algebra $\g$ spanned by
Lie derivatives $L(x)$ and contractions $I(x)$ ($x\in\g$) and satisfying the Cartan's magic formula
$[d, I(x)]=L(x)$, the corresponding current algebras coincide, 
and they are equal to $\CA(M, C\g)=\SA(M, C\g) \cong C^\infty(M, \g)$
the space of smooth $\g$-valued functions on $M$ with the pointwise Lie bracket.
Other examples include affine Lie algebras on the circle and Faddeev-Mickelsson-Shatashvili (FMS)
extensions of higher-dimensional current algebras.

The second set of results  is related to the construction of a new DGLA $D\g$ assigned
to a Lie algebra $\g$. It is generated by $L(x)$ and $I(x)$ (similarly to $C\g$) and by
higher contractions $I(x^2), I(x^3)$ {\em etc}. Similarly to $C\g$, $D\g$ can be used in building
differential models of equivariant cohomology. In particular, twisted equivariant cohomology
(including twists by 3-cocycles and higher odd cocycles) finds a natural place in this 
new framework. The DGLA $D\g$ admits a family of central extensions $D_p\g$ parametrized
by homogeneous invariant polynomials $p \in (S\g^*)^\g$. There is a Lie homomorphism
from $\CA(M, D_p\g)$ to the FMS current algebra defined by $p$.

Let $G$ be a Lie group integrating the Lie algebra $\g$.
The current algebras $\SA(M, D\g)$ and $\SA(M, D_p\g)$ integrate to groups $DG(M)$ and $D_pG(M)$.
As a topological application, we consider principal $G$-bundles,
and  for every  homogeneous polynomial $p \in (S\g^*)^\g$ we pose a lifting problem 
(defined in terms of $DG(M)$ and $D_pG(M)$)
with the only obstruction the Chern-Weil class ${\rm cw}(p)$.
In the case of $M$ a sphere, we study integration of of the current algebra $\CA(M, D_p\g)$.
It turns out that the corresponding group is a central extension of $DG(M)$. Under certain
conditions on the polynomial $p$, this is a central extension by a circle.

\end{abstract}
\maketitle

%\begin{document}

\section{Introduction}
This paper has two main themes: differential models of equivariant cohomology and
current algebras.

Our first main result is the construction of the current algebra functor
$\CA$ which associates a Lie algebra (a current algebra) $\CA(M, A)$ to a pair
of a manifold $M$ and a differential graded Lie algebra (DGLA) $A$. The functor
$\CA$ is contravariant with respect to $M$ and covariant with respect  to $A$.
As a vector space,
$$
\CA(M, A) = (\Omega(M) \otimes A)^{-1}/(\Omega(M) \otimes A)^{-1}_{\rm exact}.
$$
Here $\Omega(M)$ stands for differential forms on $M$.
The Lie bracket of $\CA(M, A)$ is defined by the derived bracket construction of \cite{yks}.

Another natural current algebra functor is given by 
$$
\SA(M,A)=(\Omega(M \otimes A)^0_{\rm closed} .
$$
If $A$ is acyclic as a complex, $\CA(M, A)$ and $\SA(M, A)$ are naturally isomorphic to each other.
In contrast to $\CA(M, A)$, the construction of $\SA(M, A)$ is local, hence it defines a sheaf of Lie algebras on $M$. 

Let $\g$ be a finite-dimensional Lie algebra, and $C\g$ be a DGLA spanned in degree zero by Lie  derivatives
$L(x)$ and in degree $-1$ be contractions $I(x)$ for $x \in \g$. The differential on $C\g$ is defined by the
Cartan's magic formula $[d, I(x)]=L(x)$. The corresponding current algebra $\CA(M, C\g)$ is isomorphic
to the space of maps $C^\infty(M,\g)$ with the pointwise Lie bracket. Other examples of current algebras
include affine Lie algebras over the circle and Faddeev-Mickelsson-Shatashvili (FMS) extensions of 
$C^\infty(M,\g)$ for ${\rm dim}(M) \geq 3$ (for the definition of FMS current algebras, see \cite{F,FS,Mickelsson}).

The DGLA $C\g$ is the basis of Cartan's construction of differential models of equivariant cohomology.
In more detail, let $G$ be a compact connected Lie group with Lie algebra $\g$ and $M$ be a manifold acted by $G$.
Then, the cochain complex $\Omega(M)$ carries a compatible $C\g$-action (defined by Lie derivatives and contractions).
By  Cartan's theorem, the equivariant cohomology $H_G(M, \R)$ coincides with the cohomology
of $(W\g \otimes \Omega(M))^{C\g}$, where $W\g=S\g^* \otimes \wedge \g^*$ is the Weil algebra.
The $C\g$-action and the differential on $W\g \otimes \Omega(M)$ can be chosen
in two different ways which are called Weil and Cartan differential models
of equivariant cohomology (see \cite{Cartan}, and \cite{GS} for a modern review).
The equivalence between the two models is established by the Kalkman map
induced by a group-like element $\phi \in W\g \otimes \mathcal{U}(C\g)$.

Our second observation is that  for $W\g \otimes \Omega(M)$ 
to carry a differential and a compatible $C\g$-action it is not necessary for
$\Omega(M)$ to be a $C\g$-module.
Instead, it may be a module under the action of a bigger
DGLA $D\g$ (see Theorem \ref{lgwg}). The latter has generators $L(x)$ and $I(x)$
(similar to $C\g$), and also contains `higher contractions' $I(x^2), I(x^3)$ {\em etc.}
In general, for every $D\g$-module $V$ we obtain a differential and a compatible
$C\g$-action on $W\g \otimes V$. One can define the equivariant cohomology of $V$
as the cohomology of $(W\g \otimes V)^{C\g}$. Again, there are two different models
of equivariant cohomology, and the equivalence is established by the Kalkman-type twist
induced by a group-like element $\Phi \in W\g \otimes \mathcal{U}(D\g)$ (see Theorem \ref{prop:phi}).
Among other things, the element $\Phi$ contains the information about chains 
of transgression for all invariant polynomials $p \in (S\g^*)^\g$ (a chain
of transgression is an element $e\in W\g$ such that $d_We=p\otimes 1$, where
$d_W$ is the Weil differential on $W\g$). 

One set of examples is provided by the theory of twisted equivariant cohomology.
In more detail, in the Cartan model of equivariant cohomology 
an equivariant cocycle on $M$ is an element
$\omega(t) \in (S\g^* \otimes \Omega(M))^\g$ which is closed under the Cartan differential $d_\g=d-I(t)$. Generators of $S\g^*$ have degree $2$, and
we view equivariant differential forms as polynomial functions
on $\g$ and put $t\in \g$. If $\omega(t)$ is an equivariant  3-cocycle, it can be written in the form
$\omega(t)=\omega_3+ \omega_1(t)$, where $\omega_3 \in \Omega^3(M)^\g$
and $\omega_1 \in (\g^* \otimes \Omega^1(M))^\g$. This allows to twist the differential and the
$C\g$ action on $\Omega(M)$ in the following way:
$$
\tilde{d}=d+\omega_3 \, , \qquad
\tilde{I}(x)=I(x)+\omega_1(x) \, , \qquad
\tilde{L}(x)=L(x).
$$
This twisted action finds its use in the theory of group valued moment maps (see {\em e.g.} \cite{AM}).
It is intimately related to twisting of the Cartan's differential
\begin{equation} \label{dtwisted}
\tilde{d}_\g=d_\g + \omega(t),
\end{equation}
and to the theory of twisted equivariant cohomology \cite{HU}.
The Cartan differential can be twisted by odd equivariant cocycles of higher degree. 
But the twist of the $C\g$ action on $\Omega(M)$ does not generalize to this case.
Instead, one should consider a twist of a certain  $D\g$-action on $\Omega(M)$.

In contrast to $C\g$, $D\g$ admits many central extensions. Under certain assumptions,
these central extensions are classified by homogeneous invariant polynomials $p\in (S\g^*)^\g$ (see Theorem \ref{thm:central}),
and we use notation $D_p\g$ for the central extension defined by the polynomial $p$.
If $p$ is of degree two, the extension descends to $C\g$,and the new Lie bracket 
of contractions is given by $[I(x),I(y)]=p(x,y) c$ (here $c$ is the central
generator). If $p$ of degree three and higher, one has to use $D\g$ to describe the corresponding central extension.
One result relating the two parts of the paper is as follows: there is a Lie homomorphism 
from the current algebra $\CA(M, D_p\g)$ to the FMS current algebra on $M$ defined by the 
invariant polynomial $p$. For $p$ of degree two and $M=S^1$, $\CA(S^1, D_p\g)$ coincides
with the standard central extension of the loop algebra $L\g=C^\infty(S^1, \g)$.

The DGLA $D\g$ is acyclic (see Theorem \ref{prop:acyclic}). Hence, the current algebra functors
$\CA$ and $\SA$ coincide, and they define a sheaf of Lie algebras $\SA(M, D\g)$.
If $G$ is a connected Lie group with Lie algebras $\g$, one can integrate
$\SA(M, D\g)$ to a sheaf of groups $DG(M)$. We define the gauge groupoid $\G(M)$ as the set
of $\g$-connections $A\in \mathcal{G}(M)=\Omega^1(M) \otimes \g$  together with gauge transformations
$A\mapsto {\rm Ad}_{g^{-1}}A+g^{-1}dg$. It turns out that the generalized Kalkman element $\Phi$
defines a morphism of groupoids $\mu: \G(M) \to DG(M)$  (see Theorem \ref{prop:mu1}). Both $\G(M)$ and $DG(M)$
admit a family of central extensions by $\Omega(M)^{2n-2}_{\rm closed}$ defined by an invariant homogeneous
 polynomial $p\in (S^n\g^*)^\g$. The morphism $\mu$ admits lifts to $\mu_p: \G_p(M) \to D_pG(M)$.
As a topological application, we consider the theory of $DG(M)$-torsors. Their isomorphism classes are classified
by the isomorphism classes of underlying principal $G$-bundles. It turns out that a $DG(M)$-torsor 
lifts to a $D_pG(M)$-torsor if and only if  the Chern-Weil class ${\rm cw}(p)$ vanishes (see Theorem \ref{torsor}).

The DGLA $D_p\g$ is not acyclic. Let $p$ be an invariant polynomial of degree $n$
and $\eta_p \in H^{2n-1}(G, \mathbb{R})$ its image under the transgression map.
We integrate the current algebra $\CA(M, D_p\g)$ in the case
of  $M=S^{2n-3}$ a sphere of dimension $2n-3$.
The resulting group is a central extension of $DG(S^{2n-3})$ by $\mathbb{R}/{\rm im}(\Pi)$ (see Theorem \ref{thm:extensionS^1}), where
$\Pi: \pi_{2n-1}(G, \mathbb{Z}) \to \mathbb{R}$ is the map defined by integration of the 
class $\eta_p$. If $G$ is a compact connected and simply connected Lie group
and $p$ is a generator of $(S^+\g)^\g$ (with the exception of some special cases for $G=SO(2k)$) ,
then ${\rm im}(\Pi) \cong S^1$ and one obtains a central extension by a circle.

The structure of the paper is as follows. In Section 2, we recall Cartan and Weil differential models of equivariant cohomology.
In Section 3, we define the DGLA $D\g$, establish the fact that it's acyclic, and show that a structure of a $D\g$-module 
on $V$ gives rise to a structure of a $C\g$-module on $W\g\otimes V$. In Section 4, we discuss central extensions of
$D\g$ and construct homomorphisms from $D\g$ to other DGLAs. In Section 5, we define and discuss properties of the
current algebra functors $\CA$ and $\SA$. In Section 6, we discuss sheaves of groups $DG(M)$ and $D_pG(M)$ and study torsors
over these sheaves of groups. In Section 7, we integrate to a group the Lie algebra $\CA(M, D_p\g)$ in the case of $M$ being
a sphere.

\vskip 0.2cm

{\bf Acknowledgements.} We are grateful to P. Bressler, E. Getzler, E. Meinrenken, J. Mickelsson, B. Tsygan and C. Vizman
for useful discussions. Our research was supported in part by the grants 
200020-126817 and 200020-129609 of the Swiss National Science Foundation.

\section{Differential models of equivariant cohomology}
In this section, we recall Cartan and Weil differential models of equivariant cohomology (for details, see \cite{GS}). For completeness, we include some proofs which resemble more
difficult proofs in other sections.

% Let $G$ be a topological group acting on a topological space $M$. By definition, the equivariant cohomology $H_G(M)$ is the cohomology of the space $(EG\times M)/G$, where $EG\to BG$ is the classifying principal $G$-bundle (defined by the condition that $EG$ be contractible). If $M\to B$ is a principal $G$-bundle,  $H_G(M)$ is canonically isomorphic to $H(B)$.

%In the case when $G$ is a compact connected Lie group, $M$ is a manifold and the $G$-action on $M$ is smooth, H. Cartan \cite{Cartan} found a way of computing  $H_G(M)$ using differential forms on $M$.
Let $\g$ be a Lie algebra.
The cone of $\g$ is the differential graded Lie algebra (DGLA) 
$C\g=\g[\varepsilon]=\g\oplus \g\varepsilon$, where $\varepsilon$ is an auxiliary variable of degree $-1$. The Lie bracket of $C\g$ is induced
by the Lie bracket of $\g$, and the differential given by $d/d\varepsilon$.
For $x\in\g$, we denote by $L(x)$ the element $x\in C\g$ and by $I(x)$ the element $x\varepsilon\in C\g$. They satisfy the standard relations $dI(x)=L(x)$, $[L(x),I(y)]=I([x,y])$, $[I(x),I(y)]=0$, $[L(x),L(y)]=L([x,y])$.

In general, a module over a DGLA $A$ is a cochain complex $V$ equipped with a
DGLA homomorphism $A\to {\rm End}(V)$. To define a $C\g$-module, one needs
$L_V(x)\in {\rm End}(V)^0$ and $I_V(x)\in {\rm End}(V)^{-1}$ verifying the defining
relations of $C\g$. A $C\g$-module is also called a $\g$-differential space. If $V$ is a graded commutative differential algebra, and the action of
$C\g$ is by derivations, one says that $V$ is a $\g$-differential algebra.

Let  $G$ be a connected Lie group with Lie algebra $\g$ and $M$ be a manifold
acted by $G$. Then, $C\g$ acts by derivations on differential forms $\Omega(M)$, $L(x)$ acting by Lie derivatives and $I(x)$ acting by contractions. 
This action turns $\Omega(M)$ into a $\g$-differential algebra.
Another basic example of a $\g$-differential algebra is the \emph{Weil algebra} 
$W\g=S\g^* \otimes \wedge \g^*$ which serves as a model for differential forms on 
the total space of the classifying $G$-bundle $EG$. 
The action of $L(x)$ is by the diagonal coadjoint action (extended to $S\g^*$ and
$\wedge \g^*$), and the action of $I(x)$ is by contractions on $\wedge \g^*$.
Let us choose a basis $e_a$ of $\g$ with structure constants $[e_a, e_b]=f_{ab}^c e_c$. 
We shall denote the generators of $\wedge \g^*$ by $\theta^a$ (this is a dual basis
in $\g^*$) and the generators of $S\g^*$ by $t^a$. The Weil differential is the unique degree one derivation
of $W\g$ such that
$$
d\theta^a = t^a - \frac{1}{2} f^a_{bc}\theta^b\theta^c.
$$
One can also choose $\theta^a$ and $d\theta^a$ as generators of $W\g$. Then,
it is identified with the Koszul algebra of the graded vector space $\g^*[-1]$. Thus,
$H^0(W\g)=\mathbb{R}$ and $H^i(W\g)=0$ for $i\geq 1$.

Sometimes it is convenient to consider a bigger DGLA $W\g \otimes C\g$ with 
Lie bracket induced by the one of $C\g$ and the differential $d=d_{W\g}+d_{C\g}$.
Consider elements $L(\theta)=\theta^a L(e_a), I(t)=t^a I(e_a) \in W\g \otimes C\g$.

\begin{prop} \label{prop:mc}
$I(t)-L(\theta) \in W\g\otimes C\g$ is a Maurer-Cartan element.
\end{prop}

\begin{proof}
On the one hand,
\begin{align*}
d_{C\g}\big(I(t)-L(\theta)\big) & = d_{C\g}\big(t^a I(e_a) - \theta^a L(e_a)\big) \\
 & =  t^a L(e_a), \\
d_{W\g}\big(I(t)-L(\theta)\big) & =  d_{W\g} \big(t^a I(e_a) - \theta^a L(e_a) \big) \\
& =  - f^a_{bc}\theta^bt^c I(e_a) - t^a L(e_a) + \frac{1}{2}f^a_{bc}\theta^b\theta^c L(e_a).
\end{align*}
Hence,
$$
d\big( I(t)-L(\theta) \big) = - f^a_{bc} \theta^bt^c I(e_a) + \frac{1}{2}f^{a}_{bc}
\theta^b\theta^c L(e_a).
$$
On the other hand,
\begin{align*}
[I(t)-L(\theta), I(t)-L(\theta)] & =  [t^bI(e_b)-\theta^bL(e_b), t^cI(e_c)-\theta^cL(e_c)] \\
& =  f_{bc}^a \theta^b\theta^c L(e_a) - 2f_{bc}^a\theta^bt^c I(e_a).
\end{align*}
In conclusion, we obtain
$$
d\big(I(t)-L(\theta)\big)=\frac{1}{2} [I(t)-L(\theta), I(t)-L(\theta)] ,
$$
as required.
\end{proof}

Let $\mathcal{U}(C\g)$ be the degree completed universal enveloping algebra
of $C\g$ equipped with the standard coproduct. Then, $W\g \otimes \mathcal{U}(C\g)$
is a Hopf algebra over $W\g$. Consider a degree zero group-like element
$\phi=\exp(-I(\theta))  \in \big(W\g \otimes \mathcal{U}(C\g)\big)^\g$. 

\begin{prop}  \label{prop:dphi}
$\phi^{-1} d\phi= -I(t) + L(\theta)$.
\end{prop}

\begin{proof}
We denote $\phi=\exp(a)$, where $a=-I(\theta)=-\theta^aI(e_a)$. We have,
$$
\phi^{-1} d\phi=\frac{1-\exp(-{\rm ad}_a)}{{\rm ad}_a} \, da = da - \frac{1}{2} [a, da],
$$
where we have used the standard formula for the derivative of the exponential map
and  have taken into account that $[a,[a,da]]=0$. We compute,
\begin{align*}
da & = -d\big( \theta^a I(e_a)\big) \\
& =  -d_{W\g} \big( \theta^a I(e_a) \big) - d_{C\g}\big( \theta^a I(e_a) \big) \\
& =  - t^a I(e_a) + \frac{1}{2} f^a_{bc} \theta^b\theta^c I(e_a) + \theta^a L(e_a)\ ,\\
-\frac{1}{2} [a,da] & =  \frac{1}{2} [\theta^b I(e_b), \theta^c L(e_c)] \\
& =  - \frac{1}{2} f_{bc}^a \theta^b\theta^c I(e_a).
\end{align*}
Hence,
$$
\phi^{-1}d\phi=-t^aI(e_a)+\theta^aL(e_a)=-I(t)+L(\theta).
$$
\end{proof}
%
%The $C\g$-action on a  $\g$-differential space $V$ is called free if $V$ admits the structure of a $W\g$-module compatible with the $C\g$-action. If $V$ is a unital $\g$-differential algebra then the action of $C\g$ is free if there is a homomorphism of
%$\g$-differential algebras $W\g \to V$. For instance, if $M\to B$ is a principal $G$-bundle,
%then the standard $C\g$-action on $\Omega(M)$  is free. The image of $\theta^a e_a \in W\g\otimes \g$
%is a principal $\g$-connection on $M$. The image of $t^ae_a$ is the corresponding curvature 2-form.
%The $C\g$-action on the Weil algebra $W\g$ is also free. Furthermore, for any unital $\g$-differential
%algebra $V$, the diagonal $C\g$-action on $W\g \otimes V$ is free. Note that we also have
%an action of $W\g \otimes C\g$ on $W\g \otimes V$, $W\g$ acting by multiplication on itself and $C\g$ acting on $V$.

For a $C\g$-module $V$, the basic subcomplex is defined as $V^{C\g}$.
If $M\to B$ is a principal
$G$-bundle, the basic subcomplex is isomorphic to $\Omega(B)$. For the Weil algebra $W\g$, the basic subcomplex is equal to $(S\g^*)^\g$ with vanishing differential. By definition, the equivariant cohomology of a $C\g$-module $V$ is 
$$
H_\g(V):= H((W\g \otimes V)^{C\g}, d_{W\g}+d_V). 
$$
This construction is usually referred to as the Weil model of equivariant 
cohomology.
If  $V$ admits the structure of a $W\g$-module compatible with the $C\g$-action then $H_\g(V)\cong H(V^{C\g}, d_V)$.
In particular, $H_\g(W\g)\cong (S\g^*)^\g$. If $G$ is a compact connected Lie
group and $M$ is a $G$-manifold,  Cartan's theorem  states that $H_\g(\Omega(M))$ is isomorphic to the equivariant cohomology $H_G(M, \R)=H((EG\times M)/G, \R)$ defined
by the Borel construction (here $EG$ is the total space of the classifying $G$-bundle).

If $M\to B$ is a principal $G$-bundle,
then every principal  connection gives rise to a homomorphism of $\g$-differential algebras $j:W\g\to\Omega(M)$. The image of $\theta^a e_a \in W\g\otimes \g$
is the connection 1-form and the image of $t^ae_a$ is the corresponding curvature 2-form. Since $j$ makes $\Omega(M)$ to a $W\g$-module, we have an  isomorphism
$$H_\g(\Omega(M))\cong H(\Omega(M)^{C\g})\cong H(B,\mathbb{R}).$$
This is in accordance with Cartan's theorem, as $(EG\times M)/G\cong EG\times B$ is homotopy equivalent to $B$.

The Kalkman map $\phi_V=\exp(-\theta_a I_V(e_a))$ is a natural automorphism
of $W\g \otimes V$. It transforms the differential and the action
of $I$'s in the following way,
\begin{align*}
I_{\rm new}(x) & =  \phi_V^{-1} (I_{W\g}(x)+I_V(x)) \phi_V=I_V(x), \\
d_{\rm new} & =  \phi_V^{-1} (d_{W\g} + d_V) \phi_V =
d_{W\g}+d_V-I_V(t)+L_V(\theta).
\end{align*}
Here in computing $d_{\rm new}$ we have used Proposition \ref{prop:dphi}. 
In this way, one obtains the Cartan model of equivariant cohomology.
In this model, $(W\g \otimes V)^{C\g}\cong (S\g^* \otimes V)^\g$ and the
differential takes the form
$$
d_\g=d_V-I_V(t).
$$

\begin{rem}
Since the $C\g$-action on the Weil algebra $W\g$ is free, we have 
$H_\g(W\g)\cong H((W\g)^{C\g})=(S\g^*)^\g$. In the Cartan model,
we obtain $H_\g(W\g)=H(S\g^*\otimes W\g, d_\g)$.
Let $p\in (S^n\g^*)^\g$. Then, the cocycle
$p\otimes 1 - 1\otimes p \in (S\g^* \otimes W\g)^\g$ belongs to the trivial cohomology
class (since $p \otimes 1 - 1 \otimes p \mapsto 1 \cdot p - p \cdot 1 =0$ under the
product map).  One can therefore find an element
$e\in (S\g^*\otimes W\g)^\g$ such that $d_\g e=p\otimes 1-1\otimes p$.

In this example, we denote the generators of $S\g^*$ by $t^a$,
and the generators of $W\g$ by $\theta^a$ and $f^a$.
For $n=2$, one can choose an element $e$ in the form
$$
e=-\left(p(t+f, \theta) - \frac{1}{6} p(\theta, [\theta, \theta]) \right),
$$
where $t=t^ae_a, f=f^ae_a$ and $\theta=\theta^ae_a$.
Putting $t=0$ yields $e_{t=0}=-p(f, \theta) + \frac{1}{6} p(\theta, [\theta, \theta])$ which 
is a primitive of $-p(f,f) \in W\g$. 
\end{rem}

\section{The DGLA $D\g$} \label{sec:Dg}

Let $\g$ be a Lie algebra. In this section, we define a new DGLA $D\g$ which can be used instead of the $C\g$ in differential models of equivariant cohomology. Roughly speaking, we are replacing $\g\epsilon=(C\g)^{-1}$ by its canonical free resolution. 

\subsection{Definition and basic properties of $D\g$}
Let $V$ be a negatively graded vector space $V=\oplus_{n<0} V^n$ with finite-dimensional
graded components $V^n$. We denote by 
$\mathcal{L}(V)$  the graded free Lie algebra generated by $V$. The graded
components of $\mathcal{L}(V)$ are also finite-dimensional.

Let $\g$ be a finite-dimensional Lie algebra and $S^+\g$ be the graded vector space
$\oplus_{n\geq1} S^n\g$ with the degree defined by formula $\deg S^n\g=1-2n$.
We define the graded Lie algebra $D\g$ as a semi-direct sum $\g \ltimes \mathcal{L}(S^+\g)$,
where elements of $\g$ have degree zero, and the action of $\g$ on $\mathcal{L}(S^+\g)$
is induced by the adjoint action. 

One can also view $D\g$ as a graded Lie algebra defined by
generators $l(x)$ for $x\in \g$ and $\I(u)$ for $u\in S^+\g$, and relations
$[l(x), l(y)]=l([x,y])$ and $[l(x), \I(u)]=\I({\rm ad}_x(u))$.
For $x\in \g$, it is convenient to introduce the generating function
$$
i(x)=\sum_{k=1}^\infty \I(x^k).
$$

\begin{rem}
Low degree graded components of $D\g$ are as follows: $D\g^{0} \cong \g$ with generators  $l(x)$, $D\g^{-1} \cong \g$ with generators $\I(x)$, $D\g^{-2} \cong S^2\g$
spanned by $[\I(x), \I(y)]$, $D\g^{-3} \cong S^2\g \oplus {\rm ker}(\g \otimes S^2\g 
\to S^3\g)$, where $S^2\g$ is spanned by $\I(xy)$, the map 
$\g \otimes S^2\g \to S^3\g$ is the symmetrization, and 
${\rm ker}(\g \otimes S^2\g \to S^3\g)$ is spanned by $[\I(x), [\I(y), \I(z)]]$
(subject to the Jacobi identity).
\end{rem}

\begin{prop}
The operator $d \in {\rm End}^1(D\g)$ defined by equations $dl(x)=0$ and 
\begin{equation}\label{eq:def}
  di(x)=[i(x),i(x)]/2+l(x)
\end{equation}
makes $D\g$ into a differential graded  Lie algebra.
\end{prop}

\begin{proof}
The defining relations of $D\g$ express the invariance of the definition
under the adjoint $\g$-action.
Since equation \eqref{eq:def} is invariant under this action,
it defines a  derivation of $D\g$

Next, we need to verify that $d^2=0$. Indeed,
\begin{multline*}
  d^2 i(x)=d\bigl([i(x),i(x)]/2+l(x)\bigr)=[di(x),i(x)]\\
  =\bigl[[i(x),i(x)]/2+l(x),i(x)\bigr]=\bigl[[i(x),i(x)],i(x)\bigr]/2=0,
\end{multline*}
where the last equality follows from the Jacobi identity.
\end{proof}

\begin{rem}
In low degrees, the differential is defined by the Cartan's magic formula
$d\I(x)=l(x)$ for $x\in \g$, and by its higher analogues such as
$$
d\I(xy)=\frac{1}{2} \, [\I(x), \I(y)] \hskip 0.3cm , \hskip 0.3cm d\I(x^3)=[\I(x), \I(x^2)].
$$
More generally, for $k\geq 2$ we have
$$
d\I(x^k)=\frac{1}{2} \, \sum_{s=1}^{k-1} \, [\I(x^s), \I(x^{k-s})].
$$
\end{rem}

\begin{prop}
There is a canonical projection of DGLAs $\pi: D\g \to C\g$ defined on generators by
$l(x) \mapsto L(x), \I(x) \mapsto I(x)$ and $\I(x^k) \mapsto 0$ for $k\geq 2$.
\end{prop}

\begin{proof}
The defining relations of $D\g$ is a subset of the defining relations of $C\g$. Hence,
$\pi$ is a homomorphism of graded Lie algebras. Then, we have $\pi(i(x))=I(x)$, and
$$
\pi(di(x))=\pi\bigl( \frac{1}{2}[i(x),i(x)]+l(x)\bigr)=\frac{1}{2}[I(x),I(x)]+L(x)=L(x)=d\pi(i(x)),
$$
as required.
\end{proof}

Note that $i$ can be also viewed as a formal map $\g[2]\to D\g$ of degree one. 
For a DGLA $A$, defining a DGLA homomorphism $D\g\to A$ is equivalent to giving a  Lie homomorphism $\tilde{l}:  \g\to A$ and a formal map $\tilde{i}: \g[2]\to A$ of degree one such that $\tilde{i}(0)=0$ and such that the identity \eqref{eq:def} is satisfied. The maps $l$ and $i$ define the tautological isomorphism $D\g \rightarrow D\g$. Another example is given by the canonical projection $\pi: D\g\to C\g$ with $\tilde{i}(x)=I(x)$ and $\tilde{l}(x)=L(x)$.

\begin{thm}\label{prop:acyclic}
As a complex,   $D\g$ is acyclic.
\end{thm}

We need the following auxiliary statement.

\begin{prop}\label{lemma}
  The cohomology of the DGLA $\mathcal{L}(S^+\g)$ with  differential defined by formula 
$di(x)=[i(x),i(x)]/2$ is equal to $\g=S^1\g\subset S^+\g$.
\end{prop}

Indeed, $\mathcal{L}(S^+\g)$  is the canonical free resolution of $\g[1]$ (for the standard reference, see \cite{GK}). 
For convenience of the reader, we include the proof due to Drinfeld \cite{dr}. 
We assume that $\g$ is a finite-dimensional Lie algebra.

\begin{proof}
Let us consider the differential graded associative algebra $\mathcal{U}(\mathcal{L}(S^+\g))=T(S^+\g)$.
We will be using the natural grading on $S^+\g$ defined by formula $\deg S^l\g=l$. With respect to this grading, the differential is of degree zero, and we thus have $T(S^+\g)=\bigoplus_{n=0}^\infty T_n$ as a direct sum of complexes.

Let $I^n$ be the standard $n$-dimensional cube, and consider the following simplicial complex representing $I^n$ modulo the boundary.
Degree $k$ simplices are labeled by surjective maps
$$\sigma: \{1,\dots,n\}\to\{1,\dots,k\}.$$
The geometric simplex labeled by $\sigma$ is singled out by conditions $(x_1, \dots, x_n) \in I^n$, $x_i=x_j$ if $\sigma(i)=\sigma(j)$, 
and $x_i \leq x_j$ if $\sigma(i) < \sigma(j)$. Informally, one can represent it by inequalities
$$
0\leq x_{\sigma^{-1}(1)}\leq x_{\sigma^{-1}(2)}\leq\dots\leq x_{\sigma^{-1}(k)}\leq 1,
$$
where $x_{\sigma^{-1}}(i)$ stands for all $x_j$ with $\sigma(j)=i$ (they are all equal to each other). 
We denote by $[\sigma]$ the simplex associated to $\sigma$. The standard boundary operator 
(modulo $\partial I^n$) has the form
$\partial [\sigma] =\sum_{i=1}^{k-1}  (-1)^{i-1} [\sigma_{i}] $,
where $\sigma_{i}(l)=\sigma(l)$ if $\sigma(l) \leq i$ and $\sigma_i(l)=\sigma(l)-1$ if $\sigma(l)>i$ (that is, $\sigma_i$ is gluing 
together the pre-images of $i$ and $i+1$). 

Denote by $C^n$ the corresponding simplicial cochain complex. We consider the basis of
simplicial cochains dual to the basis of chains formed by $[\sigma]$, and denote the basis element dual to $[\sigma]$ by $[\hat{\sigma}]$.
The differential of a degree $k$ cochain has the form $d [\hat{\sigma}] = \sum_{i=1}^k (-1)^{i-1} [\hat{\sigma}_i]$, where
$\sigma_i$ stands for the sum of all maps obtained from $\sigma$ by splitting the pre-image of $i$ into two non-empty
subsets (the new pre-images of $i$ and $i+1$). The cohomology of $C^n$ is one-dimensional, and it is concentrated in degree $n$,
$H(C^n) \cong H(I^n, \partial I^n) \cong \R[-n]$. The permutation group $S_n$ acts on $I^n$ preserving its boundary and the simplicial decomposition.
The induced action on the cohomology $H(C^n)$ is given by the signature representation.

Define a map $\zeta: (C^n \otimes \g^{\otimes n})[2n] \to T_n$ by formula
$$
\zeta \Big( [\hat{\sigma}] \otimes (a_1 \otimes \dots \otimes a_n) \Big) =
\frac{n_1! \dots n_k!}{n!} \, \prod_{i_1 \in \sigma^{-1}(1)} a_{i_1} \otimes \dots \otimes  \prod_{i_k \in \sigma^{-1}(k)} a_{i_k}  ,
$$
where $n_i = |\sigma^{-1}(i)|$. Under the grading where ${\rm deg} \, S^l\g = 1-2l$, this map is degree preserving (both sides have
degree $k-2n$).  Furthermore, it is invariant under the diagonal action of $S_n$ on $C^n$ and on $\g^{\otimes n}$. It is easy to see
that on the $S_n$-invariant subspace 
it restricts to an isomorphism $(C^n \otimes \g^{\otimes n})^{S_n}[2n]  \cong T_n$.  Moreover, this is an isomorphism
of complexes. We illustrate this statement by the following example:
let $n=2$, and let $\sigma: \{ 1,2\} \to \{ 1\}$ be the map gluing $1$ and $2$. Then, $d[\hat{\sigma}]=[e] +[s]$, where $e$ is the neutral element
of $S_2$ and $s$ is the transposition of $1$ and $2$. Choosing $a_1=a_2=x$, we compute
$$
\zeta\Big( d [\hat{\sigma}] \otimes (x\otimes x) \Big) = \zeta \Big( ([e]+[s]) \otimes (x\otimes x) \Big) = \frac{1}{2} \, ( x\otimes x + x\otimes x) = x\otimes x,
$$
and
$$
d \zeta\Big( [\hat{\sigma}] \otimes (x\otimes x) \Big) = d x^2 = x\otimes x.
$$

Thus, for the cohomology of $T_n$ we obtain
$$ 
H(T_n) \cong H\Big( (C^n \otimes \g^{\otimes n})^{S_n}[2n] \Big) \cong (H(C^n) \otimes \g^{\otimes n})^{S_n}[2n] \cong \wedge^n \g [n].
$$
Here we used the fact that $H(C^n)$ carries the signature representation in degree $n$. Note that the cohomology of 
$T(S^+\g)$ is isomorphic to $\oplus_{n=0}^\infty \wedge^n \g [n] = S(\g[1])$. Since the symmetrization map 
${\rm Sym}: S\big(\mathcal{L}(S^+\g)\big) \to T(S^+\g)$ is an isomorphism of complexes, we have
 $S\Big(H\big(\mathcal{L}(S^+\g)\big)\Big) \cong H\big(T(S^+\g)\big) $. By comparing with
$H\big(T(S^+\g)\big) \cong S(\g[1])$, we infer that $H^{-1}\big(\mathcal{L}(S^+\g)\big) \cong \g$. For dimensional reasons, $H^{-k}\big(\mathcal{L}(S^+\g)\big)$ 
vanishes for $k\geq 2$ (here we are using the fact that the dimension of $\g$ is finite).
\end{proof}

\begin{proof}[Proof of Theorem \ref{prop:acyclic}]
 Let us  consider the DGLA $D\g_s$ ($s\in\real)$, which is isomorphic to $D\g$ as a graded Lie algebra, and the differential is modified as follows,
  $$di(x)=[i(x),i(x)]/2+s l(x),\qquad dl(x)=0.$$ By Proposition \ref{lemma}, we have $H^0(D\g_0)\simeq H^{-1}(D\g_0)\simeq\g$, $H^i(D\g_0)=0$ otherwise.

  Notice now that $D\g\simeq D\g_s$ whenever $s\neq0$ (the isomorphism is given by redefining $i(t)$ to be $i(s t)$, i.e.\ by multiplying each $S^n\g$ by $s^n$). Since the cohomology cannot increase under a small deformation, we only need to check what happens in degrees $-1$ and 0. The differential $\g=(D\g_s)^{-1}\to\g=(D\g_s)^0$ is given by multiplication by $s$, hence the cohomology vanishes for $s\neq0$.
\end{proof}

\begin{rem}
Since $D\g$ is acyclic, it can be represented (as a complex) as a cone of some graded vector space, $D\g \cong CV$, where the low degree graded components of $V$ are of the form
$V^0\cong \g, V^1=0, V^2\cong S^2\g, V^3\cong {\rm ker}(\g \otimes S^2\g \to S^3\g)$
{\em etc}.
\end{rem}

For the future use, we shall consider a bigger DGLA, $W\g \otimes D\g$ with differential $d_{W\g} + d_{D\g}$ and the Lie bracket 
induced by the Lie bracket of $D\g$. Let $i(t)\in W\g \otimes D\g$ denote the element
$$i(t)=t^a\otimes\mathcal{I}(e_a)+t^at^b\otimes\mathcal{I}(e_ae_b)+t^at^bt^c\otimes\mathcal{I}(e_ae_be_c)+\dots.$$

\begin{prop} \label{prop:MC}
$i(t)-l(\theta) \in W\g \otimes D\g$ is  a Maurer-Cartan element.
\end{prop}

\begin{proof}
The proof is similar to the one of Proposition \ref{prop:mc}.
We compute directly in shorthand notation,
\begin{align*}
d_{D\g} (i(t)-l(\theta)) & =   \frac{1}{2} [i(t), i(t)] + l(t), \\
d_{W\g} (i(t)-l(\theta)) & = - [l(\theta), i(t)] - l(t) + \frac{1}{2} l([\theta, \theta]),
\end{align*}
where $l([\theta, \theta])=f_{ab}^c \theta^a \theta^b l(e_c)$. Adding up these two expressions we obtain,
$$
d (i(t)-l(\theta))= \frac{1}{2} [i(t), i(t)] - [l(\theta), i(t)] + \frac{1}{2} l([\theta, \theta]) =
\frac{1}{2} [i(t)-l(\theta), i(t)-l(\theta)],
$$
as required.
\end{proof}

\subsection{Modules over the Weil algebra and $D\g$-modules}\label{sect:modules}

A module over the DGLA $D\g$ is a cochain complex $V$ and a DGLA homomorphism $D\g\to \End(V)$. We shall denote the corresponding generating functions by $i_V(t)$ and $l_V(t)$ (they stand for $\tilde{i}$ and $\tilde{l}$ of the previous section).
The following proposition is our motivation for introducing $D\g$.

\begin{thm}\label{lgwg}
~
\begin{enumerate}[1.]
  \item
  Let $V$ be a $D\g$-module. Then, the free $W\g$-module  $U=W\g\otimes V$ endowed with the differential
  \begin{equation} 
    d_U=d_{W\g}+d_V-i_V(t)+l_V(\theta)\label{eq:mod_U_1}
  \end{equation}
  carries a  compatible $C\g$-action given by formulas
\begin{subequations}\label{eq:mod_U}
\begin{align}
   I_U(x) &=I_{W\g}(x) \\
   L_U(x)&=L_{W\g}(x) +l_V(x)
\end{align}
\end{subequations}
for $x \in \g$.

  \item
  Let $U$ be a free (that is, isomorphic to $W\g \otimes V$ for some graded vector space $V$) $W\g$-module with a compatible $C\g$-action. Suppose that one can choose an isomorphism $S\g^*\otimes_{W\g}U\cong S\g^*\otimes V$ in such a way that the action of $\g$  splits into the standard action on $S\g^*$ and an action $l_V$ on $V$. Then, $V$ is naturally a $D\g$-module and $U$ is naturally isomorphic to the $W\g$-module described in part 1.
\end{enumerate}
\end{thm}

\begin{proof}
For the first statement, note that by Proposition \ref{prop:MC} the combination $i_V(t)-l_V(\theta)$ is a Maurer-Cartan element in $W\g \otimes {\rm End}(V)$. Hence,
$d_U$ defined by equation \eqref{eq:mod_U_1} squares to zero, $d_U^2=0$. We shall also check  Cartan's formula:
\begin{align*}
[d_U, I_U(x)] & =  [d_{W\g}+d_V-i_V(t)+l_V(\theta), I_{W\g}(x)]  \\
& =  [d_{W\g}, I_{W\g}(x)] + [l_V(\theta), I_{W\g}(x)] \\
& =   L_{W\g}(x) +l_V(x),
\end{align*}
as required. It is easy to see that other relation of $C\g$ are also verified.

For the second statement, let us start with the isomorphism $S\g^*\otimes_{W\g}U\cong S\g^*\otimes V$ for which the $\g$-action splits, and denote the action of $\g$ on $V$ by $l_V$. Since $I_{W\g}(e_a)=\partial_{\theta^a}$ ($e_a$ is a basis of $\g$), there is a unique extension of this isomorphism to $U\cong W\g\otimes V$ so that $I_U(x)=I_{W\g}(x)$ (this is the so-called Kalkman trick). We thus have
 \begin{subequations}
\begin{align}
  d_U &= d_{W\g}+ \delta(\theta,t)\label{eq:d_U}\\
  I_U(x) &= I_{W\g}(x)=x^a\frac{\partial}{\partial\theta^a}\label{eq:L}\\
  L_U(x) &= [d_U,I_U(x)]= L_{W\g}(x)+x^a\frac{\partial \delta}{\partial\theta^a}  \label{eq:L_U}
\end{align}
\end{subequations}
for some $\delta\in (W\g\otimes\End(V))^1$. This implies
$$[L_U(x),I_U(y)]=I_U([x,y])+x^a y^b\frac{\partial^2 \delta}{\partial\theta^a\partial\theta^b} .$$
Hence, $\delta$ is at most linear in $\theta$'s. Moreover, from \eqref{eq:L_U} we see that the $\theta$-linear part of $\delta$ is in fact equal to $l_V(\theta)$.

Let us now write $\delta$ as $\delta(\theta,t)=d_V-i(t)+l(\theta)$, where $d_V=\delta(0,0)$ and $i(t)=\delta(0,0)-\delta(0,t)$. Then, the condition
 $d_U^2=0$ reads (using the action of $d_{W\g}$ on $t$ and $\theta$)
$$
0=d_U^2=d_V^2 -[d_V, i_V(t)] + \frac{1}{2} [i_V(t), i_V(t)] +l_V(t).
$$
Putting $t=0$ yields $d_V^2=0$, and the remaining part of the equation gives $[d_V, i(t)]=[i_V(t),i(t)]/2+l_V(t)$. Hence, $i_V$ and $l_V$
define a DGLA homomorphism $D\g \to {\rm End}(V)$, as required.
\end{proof}

\begin{rem}
Notice that the differential \eqref{eq:mod_U_1} and the action \eqref{eq:mod_U} of $C\g$  resemble the differential and the $C\g$-action in the  Cartan model of equivariant cohomology. Later in this section, we shall find a natural endomorphism of $U$ which transforms $d_U$ into $d_{W\g}+d_V$ and thus gives an analogue of the Weil model.
\end{rem}

Since $U$ is a $W\g$-module, we have $H_\g(U) \cong H(U^{C\g})$. Here $U^{C\g} \cong (S\g^* \otimes V)^\g$ with differential
$d_\g=d_V-i_V(t)$. 
Assume that the $D\g$-action on $V$ is induced by a $C\g$-action via the canonical
projection $\pi: D\g \to C\g$.
Then, $i_V(t)=I_V(t)$ and $(S\g^* \otimes V)^\g$ with differential $d_\g=d_V-I_V(t)$
is the Cartan model of $H_\g(V)$.

We shall now transform the differential \eqref{eq:mod_U_1} on $U=W\g\otimes V$ into $d_{W\g}+d_V$. Such a construction follows easily  from the fact that $W\g$ is $\g$-equivariantly contractible. Note that $W\g \otimes\mathcal{U}(\mathcal{L}(S^+\g))$ is a graded
Hopf algebra over $W\g$ with the coproduct induced by the canonical coproduct of
the enveloping algebra $\mathcal{U}(\mathcal{L}(S^+\g))$.

\begin{thm}\label{prop:phi}
  There exists a $\g$-invariant group-like element of degree zero
  $$\Phi\in W\g \otimes\mathcal{U}(\mathcal{L}(S^+\g))\subset W\g\otimes\mathcal{U}(D\g)$$
  such that
  \begin{equation}\label{eq:phidphi}
    \Phi^{-1}d\Phi=-i(t)+l(\theta).
  \end{equation}
\end{thm}

\begin{proof}
Recall the following fact: let $A$ be a DGLA and $\alpha$ be a Maurer-Cartan element in $A\otimes\Omega(I)$, where $I=[0,1]$ is the unit interval. We have $\alpha=a(s)+b(s)\,ds$, where $a(s)$ is a family of Maurer-Cartan elements in $A$ (parametrized by $s$) and $b(s)\in A^0$. Let $\Phi$ be the holonomy from $0$ to $1$ of the $A^0$-connection $b(s)\,ds$ on $I$. Then, $\Phi$ transforms $a(0)$ to $a(1)$, i.e.\ $a(1)=\Phi^{-1}a(0)\Phi - \Phi^{-1}d\Phi$.

By Proposition \ref{prop:MC},  $i(t)-l(\theta)\in W\g\otimes D\g$ is a Maurer-Cartan element. Consider the morphism of dg algebras $W\g\to W\g\otimes\Omega(I)$ given by $\theta^a\mapsto\theta^a\otimes s$, where $s$ is the coordinate on $I$. It gives rise to a morphism of DGLAs
  $$W\g\otimes D\g\to W\g\otimes D\g \otimes\Omega(I).$$
  Let $\alpha=a(s)+b(s)\,ds$ be the image of $i(t)-l(\theta)$ under this morphism. Then, we have  $a(0)=0$ and $a(1)=i(t)-l(\theta)$. The element $b(s)$ takes values in the pronilpotent subalgebra $W\g\otimes \mathcal{L}(S^+\g)$, so the holonomy $\Phi$ is well defined. This implies, $i(t)-l(\theta)=-\Phi^{-1}d\Phi$, as required.
  \end{proof}

\begin{rem}
Theorem \ref{prop:phi} should be compared to Proposition \ref{prop:dphi}.
In contrast to equation $\phi=\exp(-I(\theta))$,  the  explicit 
formula for the element $\Phi$ is more involved. For $x,y\in\g$ let $\langle \partial i(x),y\rangle $ be defined by 
$$\langle \partial i(x),y\rangle = \frac{d}{d r} i(x+ry)\big|_{r=0}.$$
Then, $\Phi$ is the parallel transport from $s=0$ to $s=1$ of the connection
$$- \, \biggl\langle  \partial i\Bigl(s\,t+ \frac{s^2-s
}{2}[\theta,\theta]\Bigr),\theta \biggr\rangle\, ds.$$
Computing the contributions up to degree 3 yields
$$\Phi=\exp(-\mathcal{I}(\theta))-\mathcal{I}\big(t\theta-\tfrac{1}{6}[\theta,\theta]\theta\big)+(\text{terms of degree }{}\geq 4).$$
\end{rem}
\vskip 0.2cm

Let us  define a $\g$-equivariant linear map $Y: \g \to (W\g \otimes D\g)^{-1}$ by formula
$$
Y(x)=-\big( I_{W\!\g}(x) \Phi\big) \Phi^{-1}.
$$
For a $D\g$-module $V$, one can define $\Phi_V \in \big(W\g \otimes {\rm End}(V) \big)^0$ and $Y_V(x) \in \big(W\g \otimes {\rm End}(V)\big)^{-1}$ 
as images of $\Phi$ and $Y(x)$ under the action map.

\begin{prop} \label{Weil}
Under the natural transformation defined by $\Phi_V^{-1}$, the $C\g$-module 
$U=W\g\otimes V$ given by \eqref{eq:mod_U_1} and \eqref{eq:mod_U} is naturally isomorphic to  $U'=W\g\otimes V$ with  differential and $C\g$-action given by
\begin{subequations}
\begin{align*}
   d_{U'}&=d_{W\g}+d_V, \\
   I_{U'}&=I_{W\g}+Y_V, \\
   L_{U'}&=L_{W\g}+l_V. 
\end{align*}
\end{subequations}
\end{prop}

\begin{proof}
Since $\Phi$ is $\g$-equivariant, we have 
$$
L_{U'}(x)=\Phi_V L_U(x) \Phi_V^{-1}=L_U(x)=L_{W\g}(x) + l_V(x).
$$
For contractions, we obtain
$$
I_{U'}(x)=\Phi_V I_U(x) \Phi_V^{-1} = \Phi_V I_{W\g}(x) \Phi_V^{-1}  = I_{W\g}(x) + Y_V(x),
$$
as required.
Finally, note that
$$
\Phi_V^{-1}(d_{W\g}+d_V) \Phi_V = d_{W\g}+d_V -i_V(t)+l_V(t)= d_U.
$$
Hence, 
$$
\Phi_V d_U \Phi_V^{-1} = d_{W\g} + d_V = d_{U'}.
$$
\end{proof}

\begin{rem}
Again, one can replace $D\g$ by $C\g$ in Proposition \ref{Weil}. Then, the map $Y$ takes the form $Y(x)=-(I_{W\g}(x)\phi)\phi^{-1} = I(x)$, and $I_{U'}(x)=I_{W\g}(x)+I_V(x)$. That is, we obtain the Weil model
of equivariant cohomology of the $\g$-differential space $V$. 
\end{rem}

\section{Central extensions and DGLA homomorphisms of $D\g$}

In this Section we study further properties of $D\g$ including central extensions
and homomorphisms from $D\g$ to other DGLAs.

\subsection{Central extensions of $C\g$}
We start with an easier problem of central extensions of $C\g$.
Let $C \to A \to C\g$ be a central extension of  $C\g$ by a graded vector space $C$.
Assume that the extension $A$ is split over $\g$, and that there is a 
$\g$-equivariant injective map 
$\tilde{I}: C\g^{-1} = \g \varepsilon \to A^{-1}$ such that the composition 
$\g \varepsilon \to A^{-1} \to  \g \varepsilon$ is the identity map . 
For instance, for $\g$ reductive these assumptions are always satisfied.

In general, central extensions are classified by the second cohomology group of $C\g$
with values in $C$. A 2-cocycle consists of a degree zero map $\omega: \wedge^2 C\g \to C$
and a degree one map $\partial: C\g \to C$. 
For the map $\omega$, note that  $[L(x), L(y)]=L([x,y])$ (the extension
is split over $\g$), and $[L(x), \tilde{I}(y)]=\tilde{I}([x,y])$ (the map $\tilde{I}$ is $\g$-equivariant).
Hence, the only nontrivial part of $\omega$ is the map $\omega: \wedge^2 \g \epsilon \to C^{-2}$.
It is easy to see that the only condition on $\omega$ is $\g$-invariance. For instance, if
$C^{-2} = \R c$, $\omega$ is defined by a degree two invariant polynomial $p \in (S^2\g^*)^\g$.
Then, $[\tilde{I}(x), \tilde{I}(y)]=-2p(x,y) c$, where the normalization is chosen to match 
the natural normalization of the next section. We shall denote this central extension by $C_p\g$.
For the map $\partial$, it is completely defined by a character $\chi: \g \to C^0$. We have
$d \tilde{I}(x)= L(x) + \chi(x)$ and $dL(x)= - d \chi(x)$. This extension is trivial since
$\tilde{I}(x)$ and $\tilde{L}(x)=L(x)+\chi(x)$ define a DGLA homomorphism
$C\g \to A$.

\subsection{Central extensions of $D\g$}
Again, let
\begin{equation}\label{eq:some-ext}
  C\to A\to D\g
\end{equation}
be a central extension of DGLAs split over $\g\subset D\g$.
Similar to the previous section, we assume that the map $i:S^+\g\to D\g$ can be lifted to a \emph{$\g$-equivariant} (grading-preserving) map $\iti:S^+\g\to A$. 
For instance, this is always true if $\g$ is reductive.
Together with the splitting over $\g$, the lift $\iti$ defines a morphism of graded Lie algebras
$s:D\g\to A$ which is a splitting of the extension \eqref{eq:some-ext}. Let $\mathcal{J}=[D\g,D\g]+\g\subset D\g$. Notice that $s|_\mathcal{J}$ does not depend on the choice of $\iti$ and is a morphism of DGLAs (unlike $s$). Since $D\g/\mathcal{J}=(S^+\g)_\g$ with vanishing bracket and differential,  central extensions of $D\g$ by $C$ are in one-to-one correspondence with extensions of complexes
$$C\to A'\to (S^+\g)_\g.$$
We have thus proved
\begin{thm}  \label{thm:central}
  The category of central extensions of $D\g$ which are split over $\g\subset D\g$ and admit a $\g$-equivariant lift of the map $i$ is equivalent to the category of extensions of the complex $(S^+\g)_\g$. In particular, extensions by a complex $C$ are classified by maps
  $$(S^+\g)_\g\to H(C)[1].$$
\end{thm}

Let us also describe these extensions at the level of cochains. Since the lift $\iti$ defines a splitting $s$ of the extension \eqref{eq:some-ext}, we have $A\cong D\g\oplus C$ as a graded Lie algebra. The differential on $A$ satisfies
\begin{equation}\label{eq:univext}
  d\iti(t)=[\iti(t),\iti(t)]/2+l(t)+q(t)
\end{equation}
for some $\g$-invariant map $q:S^+\g\to C[1]$ (which can be seen as a power series $q:\g[2]\to C$ of total degree 2, such that $q(0)=0$). The formula \eqref{eq:univext} makes $D\g\oplus C$ to a DGLA if and only if
$$dq(t)=0.$$ We shall denote this DGLA by $D\g\oplus_q C$. In particular, for $C =  \real[2n-2]$ and 
the map $q$ defined by an invariant degree $n$ polynomial $p\in (S^n\g^*)^\g$  we shall denote $D\g\oplus_q \real[2n-2]$ simply by $D_p\g$.

\begin{rem}
For $n=2$, $p \in (S^2\g^*)^\g$ defines an invariant symmetric bilinear form on $\g$.
At the level of generators, the differential of $\tilde{\I}(xy)$ is modified as follows,
$$
d\tilde{\mathcal{I}}(xy)= \frac{1}{2} [\tilde{\mathcal{I}}(x), \tilde{\mathcal{I}}(y)] +p(x,y) c.
$$
Note that this central extension descends to $C\g$ (by putting all higher generators including $\tilde{\I}(xy)$ equal zero).
The corresponding equation reads $[I(x), I(y)]=-2p(x,y) c$ giving rise to the extension $C_p\g$ (see the previous section).

For $n=3$, we choose $p\in (S^3\g^*)^\g$ and modify the differential of $\tilde{I}(x^3)$,
$$
d \tilde{\mathcal{I}}(x^3)=[\tilde{\mathcal{I}}(x), \tilde{\mathcal{I}}(x^2)] +p(x^3) c.
$$
This (and higher) extensions do not descend to $C\g$.
\end{rem}

One interesting property of $D_p\g$ is as follows. Denote by $s$ the injection
$D\g \to D_p\g=D\g \oplus_q \R[2n-2]$. Then, one can define a group-like element
$\Phi_p=({\rm id} \otimes s)\Phi \in W\g \otimes \mathcal{U}(D_p\g)$
(here $\Phi \in W\g \otimes \mathcal{U}(D_\g)$ is defined in Theorem \ref{prop:phi}).

\begin{prop} \label{prop:e_p}
$\Phi_p^{-1} d\Phi_p=\iti(t)-l(\theta)-e \otimes c$, where $c\in D_p\g$
is the generator of the central line $\R[2n-2]$, and $e \in (W\g)^\g$ is  such that $de=p$.
\end{prop}

\begin{proof}
Since $\Phi_p$ is a group-like element, $\Phi_p^{-1} d\Phi_p$ is an element of
$W\g \otimes D_p\g$, and its projection to $W\g \otimes D\g$  is equal to
$\iti(t)-l(\theta)$ (see Theorem \ref{prop:phi}). Hence,
$\Phi_p^{-1} d\Phi_p=\iti(t)-l(\theta)-e \otimes c$ for some $e\in W\g$.

Note that the expression $\Phi_p^{-1}d\Phi_p$ is automatically a Maurer-Cartan
element. This implies,
\begin{align*}
0 & =  d\Big(\Phi_p^{-1} d \Phi_p\Big) - 
\frac{1}{2} [\Phi_p^{-1} d \Phi_p, \Phi_p^{-1} d \Phi_p] \\
& = 
d(\tilde{\iota}(t) - l(\theta)-e\otimes c) - 
\frac{1}{2}[\tilde{\iota}(t) - l(\theta)-e\otimes c, \tilde{\iota}(t) - l(\theta)-e\otimes c] \\
& =  p  \otimes c - (de) \otimes c .
\end{align*}
We conclude that $de=p$, as required.
\end{proof}

\subsection{Deformations of a DGLA homomorphism}
Let $A$ be a DGLA and $e\in A^1$ be an element of degree one. Recall that the operator  $d'=d-[e,\cdot]$ defines a new differential on $A$ 
(that is, $d'$ is a derivation of the Lie bracket and ${d'}^2=0$) if and only if the element $z=de-[e,e]/2$ lies in the center of $A$. The Jacobi identity implies that $z$ is closed,
$dz=0$. Indeed,
$$
dz=d\bigl( de -\frac{1}{2} [e,e]\bigr) =\frac{1}{2}\left([e,de] - [de,e]\right)
=[e, z + \frac{1}{2} [e,e]]=\frac{1}{2}[e,[e,e]]=0.
$$
We shall denote the graded Lie algebra $A$ equipped with this new differential by $A_{(e)}$.

A DGLA homomorphism $D\g\to A$ consists of a Lie algebra homomorphism
$l_A:\g\to A$ and a $\g$-equivariant formal power series $i_A:\g[2]\to A$ of degree 1, satisfying $i_A(0)=0$ and Equation \eqref{eq:def}. 
Consider simultaneous deformations of the differential on $A$ and of the map $i_A$. 
Choose a $\g$-equivariant power series of total degree one,
$e:\g[2]\to A$, and set
\begin{align*}
  d'&=d-[e(0),\cdot]  \,  ,\\
  i'_A(t)&=i_A(t)+e(t)-e(0) \, .
\end{align*}
These formulas define a DGLA homomorphism  $D\g\to A_{(e(0))}$ if and only if
 \begin{equation}\label{eq:morph-nonc}
   d_\g e(t):= d e(t) - [i_A(t), e(t)]=\frac{1}{2}[e(t),e(t)]+z,
 \end{equation}
where  $z\in A$ lies in the center of $A$. Note that $z$ is invariant under the $\g$-action
defined by the Lie homomorphism $l_A: \g \to A$. Then, the Jacobi identity implies
that $z$ is equivariantly closed, $d_\g z=0$.

More generally, assume that
\begin{equation}\label{eq:morph-central}
  d_\g e(t)=[e(t),e(t)]/2+z(t),
\end{equation}
where $z(t)$ takes values in the center of $A$,
and $q(t)=z(t)-z(0)$ takes values in the subspace $C \subset A$. 
Then, we obtain a DGLA  homomorphism $D\g\oplus_q C\to A_{(e(0))}$.

A natural framework for constructing examples is as follows: let $B$ be a $C\g$-module 
with a chosen $C\g$-invariant element $1_B\in B^0_\text{closed}$. For example, $B$ may be a unital
(graded) commutative  $\g$-differential algebra. Consider the semi-direct sum 
$A=C\g \ltimes B[n]$, where $B[n]$ is viewed as an abelian DGLA. Solutions
of equation \eqref{eq:morph-central} are provided by the following construction:

\begin{thm}  \label{prop:esolution}
Let $c\in B^{2k}$ be a basic cocycle,
and $p\in (S^k\g^*)^\g$ be an invariant polynomial of degree $k$. Assume that
the cohomology class of the element $p\otimes 1_B + 1 \otimes c \in (S\g^*\otimes B)^\g$ vanishes in 
$H_\g(B)$, and let $e \in (S\g^* \otimes B)^\g$ be such that
$d_\g e= p\otimes 1_B + 1\otimes c$. Consider $A=C\g \ltimes B[2k-2]$.
Then $e$ is a solution of equation \eqref{eq:morph-central} with
$z(t)=p(t)1_B+ c, z(0)=c, q(t)=p(t) 1_B$ and $C=\R[2k-2] 1_B$.
\end{thm}

\begin{proof}
Since $B[2k-2] \subset A$ is an abelian DGLA, equation 
$d_\g e=p\otimes 1 + 1 \otimes c$ implies equation \eqref{eq:morph-central}
with $z(t)=p(t) 1_B +  c$. Putting $t=0$ yields $z(0)=c$.
This element is central in $A$ since $c \in B^{2k}$ is basic. 
Finally, $q(t)=z(t)-z(0)=p(t)1_B$.
\end{proof}

\begin{rem}
There is a natural DGLA homomorphism $D\g \to A$ induced by the canonical
projection $D\g \to C\g$. Elements $e$ described in Theorem \ref{prop:esolution}
define DGLA homomorphisms $D_p\g \to A$, where the central line $R[2k-2] \subset D_p\g$
maps to $\R[2k-2] 1_B \subset B[2k-2] \subset A$.

In the case of $p=0$, we obtain a DGLA homomorphism $D\g \to A$. Then, for the
construction to work one only needs a structure of a $\g$-differential vector space on $B$.
\end{rem}

\subsection{Examples}
In this Section, we consider two examples of the construction described above.

\subsubsection{Extensions of $C\g$ by differential forms}
Let $\g$ act on a manifold $M$. We choose $B=\Omega(M)$ with the natural structure
of a $\g$-differential algebra.  Let $A=C\g \ltimes \Omega(M)[m]$. If 
$\phi(t) \in (S\g^* \otimes \Omega(M))^\g$ is a degree $m+1$ equivariant cocycle, 
then the construction of the previous Section defines a DGLA homomorphism
$D\g \to A_{(\phi(0))}$. If $m=2n-2$ and $\phi$ verifies equation 
$d_\g \phi = p\otimes 1$ for $p\in (S^n\g^*)^\g$, we obtain a DGLA
homomorphism $D_p\g \to A_{(\phi(0))}$.

In more detail, write $\phi(t)=\sum_{l=0}^{n-1} \phi_l(t)$, where $\phi_l(t)$ is a homogeneous polynomial of degree $l$ with values in  degree $2n-2l-1$ forms on $M$.
Then, $\mathcal{I}(t)$ maps to $I_M(t)+\phi_1(t)$, $\I(t^k)$ mp to $\phi_k(t)$ for $2 \leq k \leq n-1$, and $\I(t^k)$ map to zero for $k>n-1$. 
For $n=2$, the images of higher contractions (with $k\geq 2$) vanish and 
the DGLA homomorphism $D_p\g\to A_{(\phi(0))}$ descends to a homomorphism
$C_p\g\to A_{(\phi(0))}$. If $n> 2$, higher contractions map to non-vanishing differential forms $\phi_k(t)$, and the DGLA homomorphism does not
descend to $C_p\g$.

The DGLA $A_{(\phi(0))}$ acts on the differential graded algebra $\Omega(M)[[s]]$, where $s$ is a formal variable of degree $2n-2$ satisfying $ds=-\phi(0)$: the action of $C\g$ is the standard
action on $\Omega(M)$ (that is, the action is trivial on $s$) and the action of $\alpha\in\Omega(M)[m]$ is via $\alpha\,\partial_s$. A homomorphism $D\g\to A_{(\phi(0))}$ therefore gives rise to an action of $D\g$ on $\Omega(M)[s]$. This action can be seen as a twist by $\phi$ of the standard action of $C\g$ on $\Omega(M)$. One can get rid of the variable $s$ using the embedding
$\Omega(M)$ to $\Omega(M)[\![s]\!]$ by $\alpha\mapsto\alpha e^s$. The new differential on $\Omega(M)$ is then $d-\phi(0)$. ($\Omega(M)$ is  only mod-2 graded and the differential and the action of $D\g$ are no longer derivations).

\subsubsection{Extensions of $C\g$ by the Weil algebra}  \label{ex:fms}
One can choose $B$ equal to the Weil algebra $W\g$. 
Recall that for  $p\in (S^n\g^*)^\g$ one can choose an element $e \in (S\g^*\otimes W\g)^\g$
such that $d_\g e=p \otimes 1 - 1\otimes p$. 
Denote the generators of $S\g^*$ by $t^a$ and the generators of $W\g$ by $\theta^a$ and $f^a$.
The element $e(t)$ (here $t\in \g$ refers to the first factor in $(S\g^*\otimes W\g)^\g$) is a solution of \eqref{eq:morph-central} (notice that $[e(t),e(t)]=0$) for $c=-p\in W\g$ and $q(t)=p(t)$. We thus get a homomorphism $D_p\g\to A_{(e(0))}$. Notice that $e(0)\in(W\g)^\g$ satisfies  $de(0)=-p$.

For $n=2$, one can choose the element $e$ in the form
$$
e(t)=-\bigl(p(t+f, \theta) - \frac{1}{6} p(\theta, [\theta, \theta]) \bigr).
$$
Since $e(t)-e(0)=-p(t,\theta)$, $i(t)$ maps to  $I(t) - p(t, \theta)$. Note that 
$$
[I(x) - p(x, \theta), I(y) - p(y, \theta)] = - 2 p(x,y),
$$
as required by the relations of $D_p\g$. 
In the case of $n=2$, higher contractions vanish and one actually obtains a 
DGLA homomorphism $C_p\g \to A_{(e(0))}$. 

For $n\geq 3$, images of some higher contractions are necessarily non-vanishing.
For $n=3$, we can work it out in more detail. Recall that
$(D_p\g)^0\cong \g$ with generators $l(x)$, $(D_p\g)^{-1}\cong \g$ with generators
$\I(x)$, $(D_p\g)^{-2}\cong  S^2\g$ spanned by $[\I(x), \I(y)]$,
$(D_p\g)^{-3}\cong S^2\g \oplus {\rm ker}(\g\otimes S^2\g \to S^3\g)$
spanned by $\I(xy)$ and $[\I(x), [\I(y), \I(z)]]$, and 
$(D_p\g)^{-4}=(D\g)^{-4} \oplus \R$, where $\R$ is the central line. 
For the algebra $A$, we have
$A^{-4}=\R$ spanned by the unit of the Weil algebra,
$A^{-3}\cong \g^*$ with generators $\theta^a$, 
$A^{-2}\cong \g^* \oplus \wedge^2 \g^*$ spanned by 
$f^a$ and $\theta^a\theta^b$, $A^{-1}\cong \g \oplus (W\g)^{3}$
with $\g$ spanned by $I(x)$, and $A^0\cong \g \oplus (W\g)^4$,
where $\g$ is spanned by $L(x)$.

The image of the homomorphism $\rho: D_p\g \to A$ is a DGLA $B$ with
nontrivial graded components $B^0\cong \g$ with generators $L(x)$, 
$B^{-1}\cong \g$ with generators $\tilde{I}(x)$, $B^{-2}=\g^*$ with generators
$\mu(\xi)$, $B^{-3}=\g^*$ with generators $\theta(\xi)$ (here $\xi\in \g^*$), and $B^{-4}=\R c$. The differential acts as $d\tilde{I}(x)=L(x), d\theta(\xi)=\mu(\xi)$.
For the Lie bracket, $L(x)$ act on other components by the adjoint and coadjoint
actions, $B^{-2}, B^{-3}, B^{-4}$ form an abelian Lie subalgebra, 
$[\tilde{I}(x), \theta(\xi)]=\theta(\ad^*(x)\xi)$, $[\tilde{I}(x), \theta(\xi)]=\xi(x) c$, and
$$
[\tilde{I}(x), \tilde{I}(y)]=2\mu(p(x,y, \cdot)).
$$
Here $\rho(\I(xy))=\theta(p(x,y,\cdot))$, and the last relation follows from
$[\I(x), \I(y)]=2 d \I(xy)$.

\section{Current algebras}

In this Section, we introduce a functor associating to a manifold $M$ and to a DGLA $A$
a Lie algebra $\CA(M,A)$. As an application, we give a new interpretation of the 
Faddeev-Mickelsson-Shatashvili (FMS) cocycles of higher dimensional current algebras.

\subsection{Current algebra functor}

Let $A$ be a DGLA. Then, the subspace of closed elements of $A$ of degree zero $A^{0}_\text{closed}\subset A^0$ is
a Lie subalgebra of $A$. Following \cite{yks}, notice that $A^{-1}$ equipped with the bracket
\begin{equation}\label{eq:leibniz}
  \{\alpha,\beta\}:=[\alpha,d\beta]
\end{equation}
is a Leibniz algebra. That is, the bracket $\{,\}$ satisfies the Jacobi identity
$$\{\alpha,\{\beta,\gamma\}\}=\{\{\alpha,\beta\},\gamma\}+\{\beta,\{\alpha,\gamma\}\}.$$
 As the symmetric part of this bracket has exact values, the quotient space
 $A^{-1}/A^{-1}_\text{exact}$ is a Lie algebra. 

\begin{prop}
There is an exact sequence of Lie algebras,
\begin{equation}  \label{exact1}
 0 \to H^{-1}(A) \to A^{-1}/A^{-1}_{\rm exact} 
\to A^{0}_{\rm closed} \to H^0(A) \to 0,
\end{equation}
where $H^{-1}(A)$ is abelian, the Lie bracket on $H^0(A)$ is induced by the Lie bracket
on $A^0$, and 
the map $A^{-1}/A^{-1}_{\rm exact} \to A^{0}_{\rm closed}$ is induced by the differential of $A$.
\end{prop}

\begin{proof}
By definition of cohomology groups $H^{-1}(A)$ and $H^0(A)$, the sequence
\eqref{exact1} is exact.
The map $H^{-1}(A) \to A^{-1}/A^{-1}_\text{exact}$ is a Lie homomorphism 
since the derived bracket \eqref{eq:leibniz} vanishes if $\alpha$ and $\beta$
are closed. The map $A^{-1}/A^{-1}_\text{exact} \to A^{0}_\text{closed}$
is a Lie homomorphism because $d\{\alpha,\beta\}=d[\alpha,d\beta]=[d\alpha,d\beta]$.
Finally, the map $A^{0}_\text{closed} \to H^0(A)$ is a Lie homomorphism by definition 
of the Lie bracket on $H^0(A)$.
\end{proof}

\begin{prop} \label{prop:exact}
Let $0\to A\to B\to C\to 0$ be an exact sequence of DGLAs. Then, there is an
exact sequence of Lie algebras,
\begin{multline}  \label{exact2}
\dots   \to  H^{-2}(A)  \to   H^{-2}(B)  \to 
 H^{-2}(C)\to  \\
          \to  A^{-1}/A^{-1}_{\rm exact}  \to  B^{-1}/B^{-1}_{\rm  exact}
 \to  C^{-1}/C^{-1}_{\rm  exact}  \to  0,
\end{multline}
where all cohomology groups are viewed as abelian Lie algebras, 
the map $H^{-2}(C) \to A^{-1}/A^{-1}_{\rm exact}$ is the composition
of the connecting homomorphism $H^{-2}(C) \to
H^{-1}(A)$ and the natural map
$H^{-1}(A) \to A^{-1}/A^{-1}_{\rm exact}$.
\end{prop}

\begin{proof}
Replace complexes $A, B$ and
$C$ by their truncations where all components of non-negative
degrees are replaced by zero. Then, \eqref{exact2} is the corresponding long exact 
sequence. Maps between cohomology groups are Lie homomorphisms 
since the corresponding Lie brackets vanish. Maps  between Lie algebras 
equipped with derived brackets are Lie homomorphisms since 
$A\to B \to C$ are homomorphisms of DGLAs.
Finally, the map $H^{-2}(C) \to A^{-1}/A^{-1}_{\rm exact}$
is a Lie homomorphism because it factors through the Lie homomorphism
$H^{-1}(A) \to A^{-1}/A^{-1}_{\rm exact}$.
\end{proof}

Note that if $C$ is acyclic, the long exact sequence \eqref{exact2} degenerates 
to a short exact sequence,
$$
0 \to A^{-1}/A^{-1}_\text{exact}  \to  B^{-1}/B^{-1}_\text{exact}
 \to  C^{-1}/C^{-1}_\text{exact}  \to  0.
$$

\begin{prop}
Let $0\to A \to B\to C\to 0$ be an exact sequence of DGLAs. Then, there is
an exact sequence of vector spaces,
$$
0  \to  A^0_{\rm closed} \to B^0_{\rm closed} \to C^{0}_{\rm closed}
   \to  H^1(A) \to H^1(B) \to H^1(C) \to \dots
$$
If $H^0(C)$ or $H^1(A)$ vanishes, it gives rise to a short exact sequence of
Lie algebras,
$$
0\to A^0_{\rm closed} \to B^0_{\rm closed} \to C^{0}_{\rm closed} \to 0.
$$
\end{prop}

\begin{proof}
Replace the complexes $A,B$ and $C$ by their truncations where all components 
with negative degrees are replaced by zero. The corresponding long exact sequence
is the one displayed in the Proposition. The connecting homomorphism
$C^0_{\rm closed}$ is a composition of the natural projection $C^0_{\rm closed} \to H^0(A)$
and the standard connecting homomorphism $H^0(A) \to H^1(C)$. 
If either $H^0(A)$ or $H^1(C)$ (or both) vanishes, this map vanishes as well
giving rise to a short exact sequence. This is a short exact sequence of Lie algebras
since $0\to A\to B\to C\to 0$ is an exact sequence of DGLAs.
\end{proof}

For a manifold $M$ and a DGLA $A$, we consider the DGLA 
$(\Omega(M) \otimes A)$, where the Lie bracket is induced by the Lie bracket of $A$,
and the differential comes from the differential of $A$ and the de Rham differential on $M$.
We define the current algebra functor as
$$
\CA(M, A)= (\Omega(M) \otimes A)^{-1}/(\Omega(M) \otimes A)^{-1}_\text{exact} .
$$
It associates a Lie algebra (a current algebra) to a pair of a manifold and a DGLA.
It is convenient to introduce a special notation
$$
\SA(M, A)=(\Omega(M) \otimes A)^0_\text{closed}.
$$
As before, we have a natural exact sequence of Lie algebras
$$
0 \to H^{-1}(\Omega(M) \otimes A) \to \CA(M, A) \to
\SA(M,A) \to H^0(M, A) \to 0.
$$
Note that if $A$ is acyclic, the exact sequence degenerates to an isomorphism
$\CA(M, A) \cong \SA(M, A)$. 

\begin{prop}
Let $A$ be a DGLA, and suppose that, as a complex, it is isomorphic to a cone
over a graded vector space $V$. That is, $A\cong CV=V[\varepsilon]$, where
$\varepsilon^2=0$, $\deg\varepsilon=-1$ and $d=d/d\varepsilon$. Then, the current algebra 
$\CA(M, A)$ is isomorphic to  $(\Omega(M) \otimes V)^0$ as a vector space. 
\end{prop}

\begin{proof}
Since $\CA(M, A) \cong \SA(M,A)=(\Omega(M) \otimes A)^0_\text{closed}$, we consider
an element $\alpha \otimes x + \beta \otimes y \varepsilon \in (\Omega(M) \otimes A)^0$.
The closedness condition reads
$$
d(\alpha \otimes x + \beta \otimes  y \varepsilon) =
d\alpha \otimes x - \beta \otimes y + d\beta \otimes y\varepsilon = 0.
$$
Hence, $x=y$ and $\beta=d\alpha$, and the projection 
$(\Omega(M) \otimes A)^{0}_\text{closed} \to (\Omega(M) \otimes V)^0$ mapping
$\alpha \otimes x + d\alpha \otimes x \varepsilon \to \alpha \otimes x$ is
an isomorphism.
\end{proof}

\begin{rem}
Note that $\alpha \otimes x + d\alpha \otimes x \varepsilon = d(\alpha \otimes x\varepsilon)$,
where $\alpha \otimes x \varepsilon \in (\Omega(M) \otimes A)^{-1}$ defines an element
of $\CA(M, A)$.
\end{rem}

The current algebra functor is contravariant with respect to $M$
and covariant with respect to $A$. 
If $0\to A\to B\to C\to 0$ is a short  exact sequence of DGLAs, we obtain from Proposition \ref{prop:exact}
an exact sequence of Lie algebras,
\begin{multline}
0 \to {\rm im} \Bigl(H^{-2}(\Omega(M) \otimes C)  \to H^{-1}(\Omega(M) \otimes A)\Bigr)\to \\
  \to  \CA(M, A)  \to  \CA(M, B)  \to  \CA(M, C)  \to  0.
\end{multline}
Again, if $C$ is acyclic, it degenerates  
to a short exact sequence of current algebras,
$$
0 \to \CA(M, A) \to \CA(M, B)\to \CA(M, C) \to 0.
$$
In many examples, $C$ is equal to $C\g$ or $D\g$. These DGLAs are acyclic, and we obtain
short exact sequences of current algebras.

If $C$ is acyclic, we have $H^0(\Omega(M \otimes C))=0$, and we
 obtain a short exact sequence of Lie algebras
$$
0 \to \SA(M, A) \to \SA(M,B) \to \SA(M, C) \to 0.
$$

%\begin{rem}
%The Lie algebra $\bar{A}(M)$ has the following generalization. Suppose that
%$N$ is a manifold with boundary, and  $M$ is a connected component (or a union of several connected components) of the boundary. Let us define the Lie algebra
%$$A(N,M)=\bigl(\Omega(N,\partial N - M)\otimes A\bigr)^{0,\,\text{closed}}/\bigl(\Omega(N,\partial N)\otimes A\bigr)^{0,\,\text{exact}},$$
%where $\Omega(N,X)$ is the kernel of the restriction map $\Omega(N)\to\Omega(X)$ (i.e.\ the cohomology of $\Omega(N,X)$ is $H(N,X;\real)$). Now we have the exact sequence
%\begin{multline*}
%0\to H^0\bigl(\Omega(N,\partial N)\otimes A\bigr)\to A(N,M)\to A(M)\to
%H^1\bigl(\Omega(N,\partial N)\otimes A\bigr)\to \\
%\to H^1\bigl(\Omega(N,\partial N - M)\otimes A\bigr)\to H^1\bigl(\Omega(M)\otimes A\bigr)\to H^2\bigl(\Omega(N,\partial N)\otimes A\bigr) \to \dots
%\end{multline*}
%and $A(N,M)\cong \bar{A}(M)$ for $N=M\times I$.
%\end{rem}

\subsection{Examples}

In this Section, we apply the  functor $\CA$ to obtain several examples of current algebras on manifolds.

\subsubsection{$A=C\g$} \label{subsec:CACg}
Let $\g$ be a Lie algebra. The cone $C\g$ is an acyclic DGLA. 
Hence, $\CA(M, C\g) \cong \Omega^0(M) \otimes \g=C^\infty(M, \g)$. It is easy to see that the Lie bracket
of $\CA(M, C\g)$ coincides with the point-wise Lie bracket on $C^\infty(M, \g)$. Indeed,
the derived bracket of two elements $f \otimes I(x), g \otimes I(y) \in (\Omega(M) \otimes C\g)^{-1}$
is given by formula,
\begin{align*}
\{ f \otimes I(x), g \otimes I(y) \}  & =  [f \otimes I(x), d( g \otimes I(y)) ] \\
& =  [f \otimes I(x), dg \otimes I(y) + g \otimes L(y)] \\
& =  fg \otimes [I(x), L(y)] = fg \otimes I([x,y]),
\end{align*}
as required.

\subsubsection{$A=D\g$} \label{subsec:CADg}
In the case of $A=D\g$, it is difficult to give a compact description of $\CA(M, D\g)$. By Proposition \ref{prop:acyclic}, 
$D\g$ is acyclic. It is a cone $CV$ over a graded vector space $V$ with $V^0 =\g, V^{-1}=0, V^{-2}=S^2\g, V^3={\rm ker}(s: \g \otimes S^2\g \to S\g^3)$ {\em etc.}
Here the map $s: \g \otimes S^2\g \to S\g^3$ is the symmetrization.
Let $\pi: D\g \to C\g$ be the natural projection. Then, the short exact sequence $0 \to {\rm ker}(\pi) \to D\g \to C\g \to 0$
gives rise to a short exact sequence of current algebras
$$
0 \to \CA(M, {\rm ker}(\pi)) \to \CA(M, D\g) \to C^\infty(M, \g) \to 0.
$$
In particular, $\CA(M, {\rm ker}(\pi))$ contains a subspace isomorphic to $\Omega^2(M) \otimes S^2\g$.
Repeating the computation of section \ref{subsec:CACg}, we obtain
\begin{align*}
\{ f \otimes \I(x), g \otimes \I(y) \}  & =  [f \otimes \I(x), d( g \otimes \I(y)) ] \\
& =  [f \otimes \I(x), dg \otimes \I(y) + g \otimes l(y)] \\
& =  fdg \otimes [\I(x), \I(y)] + fg \otimes [\I(x), l(y)] \\
& =  fdg \otimes  2d\I(xy)  +   fg \otimes \I([x,y]) \\
& =  - 2 df \wedge dg \otimes \I(xy) + fg \otimes \I([x,y]).
\end{align*}
Here the element $df \wedge dg \otimes \I(xy) \in (\Omega(M) \otimes D\g)^{-1}/{\rm exact}$ is the image of
$df \wedge dg \otimes xy \in \Omega^2(M) \otimes S^2\g$.

\subsubsection{$A=D_p\g$}
Recall that for $p \in (S^n\g^*)^\g$ we have a short exact sequence of DGLAs
$0 \to \R[2n-2]   \to D_p\g \to D\g \to 0$ which induces a short exact sequence of
current algebras
$$
0 \to \CA(M, \R[2n-2]) \to \CA(M, D_p\g) \to \CA(M, D\g) \to 0,
$$
where $\CA(M, \R[2n-2])=\Omega^{2n-3}(M)/\Omega^{2n-3}_{\rm exact}(M)$. If $M$ is a compact
connected orientable manifold of dimension $2n-3$, $\CA(M, \R[2n-2]) \cong \R$ and
we obtain a central extension of $\CA(M, D\g)$ by a line.

For $n=2$, one can choose $M=S^1$. In this case, $\CA(M, D\g)=\CA(M, C\g)$ for 
dimensional reasons. Redoing again the calculation of the previous two sections, we obtain
\begin{align*}
\{ f \otimes I(x), g \otimes I(y) \}  & = [f \otimes I(x), d( g \otimes I(y)) ] \\
& =  [f \otimes I(x), dg \otimes I(y) + g \otimes L(y)] \\
& =  -2 fdg \otimes p(x,y) c + fg \otimes I([x,y]).
\end{align*}
The isomorphism $\Omega^{1}(S^1)/\Omega^{1}_{\rm exact}(S^1) \cong \R$ 
is given by the integral of a 1-form over the circle. Hence, the cocycle term
in the Lie bracket reads $-2p(x,y) \int \, fdg$ which coincides (up to normalization)
with the standard Kac-Moody central extension of the loop algebra.

For $n=3$, we choose $M$ to be a compact orientable 3-manifold. 
In this case, 
$$
\CA(M, D\g)=C^\infty(M, \g) \oplus \big(\Omega^2(M) \otimes S^2\g\big)
\oplus \big(\Omega^3(M) \otimes {\rm ker}(\g \otimes S^2\g \to S^3\g)\big).
$$
Here the map $\g \otimes S^2\g \to S^3\g$ is the symmetrization.
The computation of the Lie bracket elements of $C^\infty(M, \g)$ is exactly the same as in section \ref{subsec:CADg}.
However, there is a new feature in the following Lie bracket,
\begin{align*}
\{ \alpha \otimes \tilde{\I}(xy), f \otimes \tilde{\I}(z)\} & =  [\alpha \otimes \tilde{\I}(xy), d(f \otimes \tilde{\I}(z))] \\
& =  [\alpha \otimes \tilde{\I}(xy), df \otimes \tilde{\I}(z) + f \otimes l(z)] \\
& =  -f \alpha \otimes \tilde{\I}(\ad_z(xy)) + \alpha \wedge df \otimes 
[\tilde{\I}(xy), \tilde{\I}(z)].
\end{align*}
The last Lie bracket is of the form
\begin{multline*}
[\tilde{\I}(xy),\tilde{\I}(z)]  =  
\frac{1}{3} \left(2[{\I}(xy), \tilde{\I}(z)] - [\tilde{\I}(xz), \tilde{\I}(y)]
- [\tilde{\I}(yz), \tilde{\I}(x)] \right) \\
 +  
 \frac{1}{3}\left( [\tilde{\I}(xy), \tilde{\I}(z)] + [\tilde{\I}(xz), \tilde{\I}(y)]
+ [\tilde{\I}(yz), \tilde{\I}(x)] \right), 
\end{multline*}
where the first term is an element of 
${\rm ker}(\g \otimes S^2\g \to S^3\g)$, and the second term
can be represented as
$$
 \frac{1}{3}\left( [\tilde{\I}(xy), \tilde{\I}(z)] + [\tilde{\I}(xz), \tilde{\I}(y)]
+ [\tilde{\I}(yz), \tilde{\I}(x)] \right) = \tilde{\I}(xyz) - p(xyz)c.
$$
Again, the isomorphism $\Omega^3(M)/\Omega^{3, {\rm exact}}(M) \cong \R$ is given by the integral over $M$,
and the new cocycle term reads $- p(xyz) \int_M \alpha \wedge df$.

\subsubsection{Faddeev-Mickelsson-Shatashvili current algebra}
Recall Section \ref{ex:fms}:  let $p\in (S^n\g^*)^\g$, $e\in (S\g^*\otimes W\g)^g$ such that $d_\g e=p\otimes 1 - 1 \otimes p$,
and let $A_{FMS}=\big(C\g\ltimes W\g[2n-2]\big)_{(e(0))}$ be the semi-direct product of $C\g$ and $W\g[2n-2]$ with differential $d'=d-[e(0), \cdot ]$. 
In fact, the only part of the differential which is changed is $d' I(x)= L(x) - I_{W\g}(x) e(0)$.

The differential on $W\g[2n-2]$ is induced by the Weil differential. Hence, the embedding $\mathbb{R}[2n-2] \to W\g[2n-2]$ is a chain map
(and a homomorphism of abelian DGLAs). As a consequence, we obtain a short exact sequence of DGLAs $0 \to \R[2n-2] \to A_{FMS} \to A' \to 0$,
where $A'=\big(C\g\ltimes W^+\g[2n-2]\big)_{(e(0))}$ is an acyclic DGLA. Then, 
we obtain a
short exact sequence of current algebras
$$
0 \to \Omega^{2n-3}(M)/\Omega^{2n-3}_{\rm exact}(M) \to \CA(M, A_{FMS}) \to \CA(M, A') \to 0.
$$
If $M$ is a compact connected orientable manifold of dimension $2n-3$, we have $\Omega^{2n-3}(M)/\Omega^{2n-3}_{\rm exact}(M) \cong \R$
with an isomorphism defined by integration. 

One can view the abelian current algebra $\CA(M, W^+\g[2n-2])$ as a space of local
functional of $\g$-connections on $M$. 
Recall that a $\g$-connection on $M$ defines a homomorphism of graded
commutative algebras $W\g \to \Omega(M)$.
Taking into account that 
$\bigl(W^+\g[2n-2] \otimes\Omega(M)\bigr)^{-1} \cong \bigl(W^+\g\otimes\Omega(M)\bigr)^{2n-3}$, we obtain (for each $\g$-connection) a map
$$
\CA(M, W^+\g[2n-2])\to \Omega(M)^{2n-3}/\Omega(M)^{2n-3}_{\rm exact} \cong \R.
$$
The Lie algebra $\CA(M, A')$ is therefore an abelian extension of $\CA(M, C\g)=C^\infty(M, \g)$
by functionals on the space of $\g$-connections.

Computing the Lie bracket of the elements $f \otimes x, g \otimes y \in C^\infty(M, \g)$ yields
\begin{align*}
\{ f \otimes I(x), g \otimes I(y) \}  & =  [f \otimes I(x), d( g \otimes I(y)) ] \\
& =  [f \otimes I(x), dg \otimes I(y) + g \otimes L(y) - I_{W\g}(y) e(0)] \\
& =  fg \otimes (I([x,y]) - I_{W\g}(x)I_{W\g}(y) e(0)).
\end{align*}
The cocycle term reads $ - fg \otimes I_{W\g}(x)I_{W\g}(y) e(0)$. 
For $n=2$, it reads $- fg \otimes p([x,y], \cdot)$. For higher $n$,
this formula defines the Faddeev-Mickelsson-Shatashvili (FMS) cocycle \cite{F,FS,Mickelsson} on the
Lie algebra of maps from $M$ to $\g$ with values in local functionals of $\g$-connections.

\subsubsection{Truncated FMS current algebra}
The construction of Section \ref{ex:fms} defines a DGLA homomorphism
\begin{equation}\label{eq:B}
\rho: D_p\g\to \bigl(C\g\ltimes W\g[2n-2]\bigr)_{(e(0))}=A.
\end{equation}
Hence, we obtain an induced homomorphism of current algebras $\CA(M, D_p\g) \to \CA(M, A_{FMS})$.
Note that the map $\rho$ restricts to the identity mapping the central line $\R[2n-2] \subset D_p\g$ to
the line $\R[2n-2] \subset W\g[2n-2]$. As a consequence, we obtain an induced homomorphism 
$\rho': D\g \to A'=A_{FMS}/\R[2n-2]$ and a homomorphism of the corresponding current algebras
$\CA(M, D\g) \to \CA(M, A')$.

The image of the map $\rho$ is a DGLA $B_{FMS}\subset A_{FMS}$. 
We refer to $\CA(M, B_{FMS})$ as
to truncated FMS current algebra. We work out in detail the example of $n=3$.
In this case, we have an exact sequence of DGLAs $0\to \R[4] \to B_{FMS} \to B' \to 0$,
where $B'=B_{FMS}/\R[4]$ is acyclic. Hence, we obtain an exact sequence of current 
algebras
$$
0 \to \Omega^{3}(M)/\Omega^{3}_{\rm exact}(M) \to \CA(M, B_{FMS}) \to \CA(M, B') \to 0,
$$
where $\CA(M, B') = C^\infty(M, \g) \oplus (\Omega^2(M) \otimes \g^*)$. 
For the Lie bracket between elements $f \otimes x, g\otimes y \in C^\infty(M,\g)$, we have
\begin{align*}
\{ f\otimes \tilde{I}(x), g\otimes \tilde{I}(y)\} & = 
[ f \otimes \tilde{I}(x), d(g \otimes \tilde{I}(y))] \\
& =  [f \otimes \tilde{I}(x), dg \otimes \tilde{I}(y) + g \otimes L(y) ] \\
& =  fdg \otimes [\tilde{I}(x), \tilde{I}(y)] + fg \otimes [\tilde{I}(x), L(y)] \\
& =  2 fdg \otimes \mu(p(x,y, \cdot) + fg \otimes \tilde{I}([x,y]) \\
& =  2df \wedge dg \otimes \theta(p(x,y, \cdot)) + fg \otimes \tilde{I}([x,y]).
\end{align*}
Here the term $2df \wedge dg \otimes \theta(p(x,y, \cdot))$ is another representative
of the FMS cocycle.

\section{Groups integrating current algebras $\SA(M, A)$}

In this Section, we construct sheaves of groups integrating sheaves of Lie algebras
$\SA(M, A)$, and in particular we apply this technique to $\SA(M, D\g)$ and $\SA(M, D_p\g)$.

%\subsection{Sheaf of Lie algebras $A(M)$ and sheaf of groups $\mathcal{A}(M)$}
%The construction of $\SA(M, A)=(\Omega(M) \otimes A)^{0}_{\rm closed}$ defines a sheaf of Lie algebras on $M$. To an open subset $U \subset M$  we associate 
%a Lie algebra  $\SA(U, A)=(\Omega(U) \otimes A)^0_\text{closed}$. The restriction maps are well-defined, and it is easy to see that they are Lie homomorphisms. 
%We will denote this sheaf of Lie algebras by $A(M)$.
%In particular, for every manifold $M$ we obtain sheaves of Lie algebras $C\g(M)$ (this is simply the sheaf of $\g$-valued functions on $M$)
%and $D\g(M)$.
\newcommand{\SG}{\mathcal{SG}}
\subsection {The group $\SG(M,A,G)$ integrating the current algebra $\SA(M,A)$}

%Let $A$ be a DGLA. It is our goal to integrate the Lie algebra $\SA(M,A)$ to a group. 
%Informally speaking, we integrate the DGLA $A$ to a dg-group $\mathcal{A}$ and then %$\SG(M,A)$ is defined as the group of dg-maps $T[1]M\to\mathcal{A}$.
%Since $\SA(M,A)=\SA(M,A_\text{trunc})$, where
%$$A^i_\text{trunc}=
%  \begin{cases}
%    A^i & i<0\\
%    A^0_\text{closed} &i=0\\
%    0 & i>0
%  \end{cases}\ ,
%$$
%we can replace $\mathcal{A}$ with $\mathcal{A}_\text{trunc}$, where %$\mathcal{A}_\text{trunc}$ is a dg-group integrating $A_\text{trunc}$.

%Let us now describe $\SG(M,A)$ more formally. 
To simplify the task, we first consider  a simpler Lie algebra 
$\overline\SA(M,A)=(\Omega(M) \otimes A)^0$.
Denote $A^0_\text{closed}=\g$,
let $G$ be a connected Lie group with Lie algebra $\g$, and let $A'=\oplus_{n<0} A^n$ be  the sum of 
graded components of $A$ of negative degrees. Suppose that the adjoint action of $\g$ on $A'$ lifts to an action of $G$.
Denote by 
$$
G(M)=C^\infty(M, G)
$$ 
the group of smooth maps from $M$ to $G$ (with pointwise multiplication).
Let  $\mathcal{U}(A')$ be the degree completion of the universal enveloping
algebra of $A'$, and let
$$
\mathcal{H}(M) = \Omega(M) \otimes  \mathcal{U}(A')
$$ 
be the Hopf algebra with coproduct
induced by the one of $\mathcal{U}(A')$. Group-like elements of degree zero 
in $\mathcal{H}(M)$ form a group 
$$
H(M)=\exp\bigl((\Omega(M)\otimes A')^0\bigr) \subset \mathcal{H}(M).
$$
The group $G(M)$ acts on $\mathcal{H}(M)$ by Hopf
algebra automorphisms. This action induces an action of $G(M)$ on $H(M)$ by group 
automorphisms. We define 
$$
\overline\SG(M,A,G) = G(M) \ltimes H(M)
$$ 
as the semi-direct
product of $G(M)$ and $H(M)$. Note that the group $\overline\SG(M,A,G)$ depends
of the choice of the connected Lie group $G$ integrating the Lie algebra
$\g=A^0_{\rm closed}$.

It is easy to see that $\overline\SG(M,A,G)$
is an integration of $\overline\SA(M,A)$. Indeed, let $(h_t, g_t)$ be a one-parameter subgroup
of $\overline\SG(M,A)$, where $h_t \in H(M), g_t\in G(M), t\in \R$. Then, its
derivative at $t=0$ is a pair $(u, x)$, where $x \in C^\infty(M, \g)$ and
$u\in (\Omega(M) \otimes A')^0$. The pair $(u,x)$
defines an element of $(\Omega(M) \otimes A)^0$. In the other direction,
every such an element can be exponentiated to a one-parameter subgroup 
of $\overline\SG(M,A,G)$.

\begin{rem}
Let $A=C\g$. Note that $C\g^0_{\rm closed} =C\g^0=\g$.
Let $G$ ba a connected Lie group integrating $\g$.
We have $A'=\g \varepsilon \subset C\g$, and
$H(M)=\{ \exp(u); \,\, u\in \Omega^1(M) \otimes \g\varepsilon\}$.
Elements of $\overline\SG(M,C\g,G)$ are of the form $hg$, where 
$h\in H(M)$ and $g\in G(M)$.
\end{rem}

The Lie subalgebra 
$\SA(M,A) \subset \overline\SA(M,A)$ is singled out by the closedness condition.
That is, a pair $(u, x) \in \overline\SA(M,A)$ belongs to $\SA(M,A)$ if $du+dx=0$. 
The analogue of this condition at the group level is as follows: 
\begin{equation} \label{eq:dgg}
\SG(M, A, G)= \{ (h,g) \in \overline\SG(M,A); \,\, h^{-1}dh+dg\,g^{-1}=0\in\Omega(M)\otimes A \} .
 \end{equation}
 Note that
 $$
 h^{-1}dh+dg\,g^{-1}=h^{-1} d(hg) \, g^{-1},
 $$
and equation \eqref{eq:dgg}  expresses the fact that $hg\in\SG(M,A, G)$ is closed.
For the product $hg=(h_1g_1)(h_2g_2)$, we have
$$
h^{-1}d(hg)g^{-1}=(g_1h_2g_1^{-1}) \left(h_1^{-1}d(h_1g_1)g_1^{-1}\right) (g_1h_2g_1^{-1})^{-1}
+g_1 \left(h_2^{-1}d(h_2g_2)g_2^{-1}\right) g_1^{-1} .
$$
Hence, $\SG(M,A, G)$ is indeed a subgroup of $\overline\SG(M,A,G)$.
%The definition of $\mathcal{A}(M)$ is local. Hence, it defines a sheaf of groups over $M$.

Again, it is easy to see that $\SG(M,A,G)$ is an integration of $\SA(M,A)$.
Indeed, let $(h_t, g_t) \in \SG(M,A,G)$ be a one-parameter subgroup,
and let $(u,x)\in \overline\SA(M,A)$ be its derivative at $t=0$. Then, equation \eqref{eq:dgg}
implies $du+dx=0$. In the other direction, for $\exp(t(u+x))=h_tg_t$ we have
$$
h_t^{-1}d(h_tg_t)g_t^{-1} =   {\rm Ad}_{g_t^{-1}} 
\frac{1-\exp(-t\, {\rm ad}_{u+x})}{ t\, {\rm ad}_{u+x}} \, d(u+x) =0
$$
if $du+dx=0$. Here we have used the standard expression for the derivative
of the exponential map.

Note that for every manifold $M$ the functor $\SA(\cdot, A)$ produces a sheaf of Lie algebras, 
where the Lie algebra of sections over $U\subset M$ is defined as $\SA(U, A)$. 
Similarly, $\SG(\cdot, A,G)$ defines a sheaf of groups.

\begin{rem}
For $A=C\g$, we consider equation \eqref{eq:dgg},
where $h=\exp(u)$ and $ u \in \Omega^1(M) \otimes \g\varepsilon$. 
Note that 
$$
h^{-1}dh=e^{-u}de^u=\frac{1-\exp(-{\rm ad}_u)}{{\rm ad}_u} \, du
=du - \frac{1}{2} [u,du].
$$ 
Here we have used that ${\rm ad}_u^k=0$ for $k\geq 2$. Let 
$e_a$ be a basis of $\g$. Then, $u=u^a \otimes e_a\varepsilon$,
$du=du^a \otimes e_a \varepsilon - u^a \otimes e_a$, and
$$
du-\frac{1}{2} [u,du]= - u^a \otimes e_a + \left(du^a\otimes e_a - \frac{1}{2} 
[u^a \otimes e_a, u^b \otimes e_b] \right) \varepsilon.
$$
Denote $\hat{u}=u^a \otimes e_a \in \Omega^1(M)\otimes \g$, and compute
$$
h^{-1}dh+dgg^{-1}=(dgg^{-1} - \hat{u}) + 
\left( d\hat{u} - \frac{1}{2} \, [\hat{u},\hat{u}]\right) \varepsilon .
$$
Hence, a pair $(\exp(u), g)$ defines an element of $\SG(M, C\g,G)$ if and only if
$\hat{u}= dgg^{-1}$ and $\hat{u}$ is a Maurer-Cartan element. The second condition
follows from the first one since $d(dgg^{-1})=(dgg^{-1})^2=[dgg^{-1},dgg^{-1}]/2$.

In conclusion, $h$ is uniquely determined by $g$,
and  the forgetful map $(h,g)\mapsto g$ defines a group isomorphism
$\SA(M,C\g,G) \cong G(M)$. The inverse map $G(M)\to\SA(M,C\g,G)$  is given by $g\mapsto(\exp(I(dgg^{-1})),g)$.
\end{rem}

\begin{rem}
Let $(A, G)$ be a pair, where $A$ is a DGLA and $G$ is a connected Lie group 
integrating $\g=A^0_{\rm closed}$ such that the adjoint action of $\g$ on $A$ lifts to an action of $G$ on $A$. We define a morphism of such pairs $(A,G)\to (B,H)$
as pairs of a DGLA homomorphism $A \to B$ and a group homomorphism 
$G \to H$ integrating the Lie algebra homomorphism $A^0_{\rm closed} \to
B^0_{\rm closed}$. Then, a morphism of pairs $(A,G) \to (B,H)$ induces
a group homomorphism $\SG(M, A,G) \to \SG(M, B,H)$. In particular, if
$A^0_{\rm closed}=B^0_{\rm closed}$ and $G=H$, we obtain a canonical
group homomorphism.
\end{rem}

\subsection{The group $C_pG(M)=\SG(M,C_p\g,G)$}
Let $p\in (S^2\g^*)^\g$. The group $C_pG(M):=\SG(M,C_p\g,G)$ is contained in the preimage of $\SG(M,C\g,G)$ under the projection map
$\overline{\SG}(M,C_p\g,G)\to \overline{\SG}(M,C\g,G)$. Therefore, it consist of elements
of the form
$$
\Bigl(h=\exp\bigl(\omega \otimes c+\tilde I(dgg^{-1})\bigr),\, g\Bigr),
$$ 
where $g:M\to G$, $\omega\in\Omega^2(M)$, and $h^{-1}dh + dgg^{-1}=0$. 
A straightforward calculation (see e.g. Proposition 5.7 in \cite{AM}) shows that
$$
h^{-1}dh+dg\,g^{-1}=(d\omega+g^*\eta_p) \otimes c,
$$
where $\eta_p\in\Omega^3(G)$ is the Cartan 3-form
(that is, a bi-invariant differential form on $G$ defined by the map
$(x,y,z)\mapsto p(x,[y,z])$ at the group unit).
The group $C_pG(M)$ can therefore be identified with the set of pairs
$$
\bigl(g:M\to G,\,\omega\in\Omega^2(M)\bigr)\text{ such that }d\omega+g^*\eta_p=0.
$$
Since $\exp(\tilde I(u))\exp(\tilde I(v))=\exp(\tilde I(u+v)+p(u,v) c/2)$, the group law is expressed in terms of these pairs,
$$(g_1,\omega_1)(g_2,\omega_2)=\bigl(g_1g_2,\omega_1+\omega_2+\frac{1}{2}(g_1\times g_2)^*\rho_p\bigr)$$
where $\rho_p\in\Omega^2(G\times G)$ is defined by $\rho_p=p(\pi_1^*\theta^L,\pi_2^*\theta_R)$ with $\theta^L$ and $\theta^R$  left-invariant and right-invariant
Maurer-Cartan forms of $G$ and $\pi_{1,2}: G\times G \to G$ projections
on the first and second factor, respectively.

\subsection{The group $DG(M)=\SG(M, D\g,G)$}
In this Section, we consider the group $\SG(M,D\g,G)$. We will use a
shorthand notation $DG(M):=\SG(M,D\g,G)$.

Observe that the exact sequence of DGLAs
$$
0 \to {\rm ker}(D\g \to C\g) \to D\g \to C\g \to 0
$$ 
gives rise (since $C\g$ is acyclic) to an exact sequence of current algebras
$$
0\to \SA(M, {\rm ker}(D\g \to C\g)) \to \SA(M, D\g) \to \SA(M, C\g) \to 0,
$$
which in turn lifts to an exact sequence of groups
$$
1 \to \SG(M, {\rm ker}(D\g \to C\g),1) \to DG(M) \to G(M).
$$
Here $1$ stands for the trivial group (${\rm ker}(D\g \to C\g)^0=0$), 
we have used  that $D\g^0\cong C\g^0 =\g$, and 
in both cases we have chosen the same connected Lie group $G$
integrating $\g$.

\subsubsection{The group $DG(M)$ and $\g$-connections}
We begin by observing  an interesting relation between the group $DG(M)$ and the space 
$\mathcal{G}(M)=\Omega^1(M) \otimes \g$ of $\g$-connections on $M$.
Let $\widehat{\mathcal{G}}(M)$ denote the action groupoid of $G(M)$ on $\mathcal{G}(M)$,
where the action is by gauge transformations,
$$
g: A \mapsto A^g={\rm Ad}_{g^{-1}} A + g^{-1}dg.
$$
Again, constructions of $\mathcal{G}(M)$ and of the gauge action are local,
and we obtain a sheaf of groupoids over $M$.

There is a natural morphism of sheaves of groupoids $\G(M)\to G(M)$ by forgetting a connection
(here we view the group $G(M)$ as a groupoid with an object set consisting of one point).
Recall that a $\g$-connection 
$A\in \mathcal{G}(M)$ gives rise to a homomorphism of $\g$-differential algebras $W\g\to\Omega(M)$, defined by $\theta\mapsto A$, $t\mapsto F_A=d A+[A,A]/2$. Under this map, the image of an element $\alpha(\theta,t)$ is  $\alpha(A,F_A)$.
Let $g\in G(M)$, $A\in \mathcal{G}(M)$, and $m_{g,A}:A\to A^g$ the corresponding morphism in  $\G(M)$. We define a map $\mu: \G(M) \to \overline\SG(M,D\g,G)$
given by the following formula,
$$
\mu(m_{g,A})=\Phi(A,F_A)\,g\,\Phi(A^g, F_{A^g})^{-1},
$$
where $\Phi$ is defined in Theorem \ref{prop:phi}. 

\begin{thm} \label{prop:mu1}
The map $\mu: \G(M) \to \overline\SG(M,D\g,G)$ is a morphism of groupoids.
It takes values in $DG(M)=\SG(M,D\g,G)$, and its composition with the natural projection
$DG(M) \to G(M)$ coincides with the forgetful map $m_{g,A} \mapsto g$.
\end{thm}

\begin{proof}
To simplify notation, we denote $\Phi(A)=\Phi(A, F_A)$.
For the composition of morphisms, we have
\begin{align*}
\mu(m_{g,A}) \mu(m_{h, A^g}) & =  
\Big( \Phi(A) g \Phi(A^g)^{-1} \Big) \Big( \Phi(A^g) h \Phi((A^g)^h) \Big) \\
& =  \Phi(A) (gh) \Phi(A^{gh}) \\
& =  \mu(m_{gh, A}) \\
& =   \mu(m_{g, A} \circ m_{h, A^g}).
\end{align*}
Hence, $\mu$ is a morphism of groupoids.

Next, we verify that $\mu(m_{g,A})$ is indeed an element of $DG(M)$.
We compute,
\begin{align*}
d(\mu(m_{g,A})) & =  \Phi(A)\left( (-\iota(F_A)+l(A))g +dg - g(-\iota(F_{A^g})
+l(A^g)) \right) \Phi(A^g)^{-1} \\
& =  \Phi(A) g \left( \iota(F_{A^g}) - \iota({\rm Ad}_{g^{-1}} F_A)+
 l({\rm Ad}_{g^{-1}} A + g^{-1}dg - A^g) \right) \Phi(A^g)^{-1} \\
& =  0,
\end{align*}
where we have used equation  \eqref{eq:phidphi}.

Finally, by compositing $\mu$ with the projection map $\pi: DG(M) \to G(M)$ we obtain
\begin{align*}
(\pi \circ \mu)(m_{g, A}) & =  \pi(\Phi(A) g \Phi(A^g)^{-1}) \\
& =  \pi(\Phi(A) {\rm Ad}_g(\Phi(A^g)^{-1}) g) \\
& =  g,
\end{align*}
as required.
\end{proof}

\subsubsection{Structure of $DG(M)$}
Using the results of the previous section, we can now prove the following proposition.

\begin{prop}
There is an exact sequence of groups
$$
1 \to \SG(M, {\rm ker}(D\g \to C\g),1) \to DG(M) \to G(M) \to 1,
$$
and the map $g \mapsto \mu(m_{g,0})$ is a section of the natural projection $DG(M) \to G(M)$.
\end{prop}

\begin{proof}
By Theorem \ref{prop:mu1}, the map $G(M) \to DG(M) \to G(M)$ defined by composing  the map
$g \mapsto \mu(m_{g,0})$ and the natural projection $DG(M) \to G(M)$ is the identity map. Hence, the projection
$DG(M) \to G(M)$ is surjective, and the sequence of groups in the Proposition is exact.
\end{proof}

Note that the DGLA ${\rm ker}(D\g \to C\g)$ is  acyclic  and  negatively graded. As a complex, 
it can be represented as a cone $CU=U[\varepsilon]$ for some 
negatively graded vector space $U$. The corresponding current algebra $\SA\bigl(M,\ker(D\g\to C\g)\bigr)$ is isomorphic (as a vector space) to
$$
\SA\bigl(M,\ker(D\g\to C\g)\bigr) \cong \bigoplus_{i>0}\Omega^i(M)\otimes U^{-i} .
$$ 
Using the grading induced by the degree of differential forms, we infer that 
this Lie algebra is nilpotent.
Hence, by composing with the exponential map we obtain a bijection
$$\nu:\bigoplus_{i>0}\Omega^i(M)\otimes U^{-i}\to \SG\bigl(M,\ker(D\g\to C\g),1\bigr).$$

\begin{prop}
The map $(g,u)\mapsto \mu(m_{g,0})\,\nu(u)$ defines a bijection
\begin{equation}\label{eq:explDg}
G(M)\times(\bigoplus_{i>0}\Omega^i(M)\otimes U^{-i})\to DG(M).
\end{equation}
\end{prop}

\begin{proof}
This follows from the facts that the map
$g \mapsto \mu(m_{g,0})$ defines a section of the projection $DG(M) \to G(M)$,
and that the map $\nu$ is bijective.
\end{proof}

\def\cs{\mathit{cs}}
\subsection{The group $D_pG(M)=\SG(M, D_p\g,G)$} In this Section, we study the group
$\SG(M, D_p\g,G)$, the shorthand notation is $D_pG(M)$.

As before, the short exact sequence of DGLAs
$$
0 \to {\rm ker}(D_p\g \to C\g) \to D_p\g \to C\g \to 0
$$
gives rise to a short exact sequence of current algebras
$$
0\to \SA(M, {\rm ker}(D_p\g \to C\g)) \to \SA(M, D_p\g) \to \SA(M, C\g) \to 0,
$$ 
which lifts to an exact sequence of groups
$$
1 \to \SG(M, {\rm ker}(D_p\g \to C\g),1) \to D_pG(M) \to G(M).
$$
In this case, the natural projection $D_pG(M) \to G(M)$ may no longer be
surjective. Again, an important tool in studying this question is the 
gauge groupoid.

\subsubsection{Central extensions of the groupoid $\G(M)$}
Let $p\in (S^n\g^*)^\g$ be an invariant polynomial of degree $n\geq 2$, and let $e_p\in (W\g)^\g$ be such that $d\, e_p=p$. The ambiguity in the choice of $e_p$ is by a closed $\g$-invariant element of $W\g$ of degree $2n-2$. Since 
$W\g$ is acyclic, for any other primitive $e'_p$ we have $e'_p-e_p=df$.

Let us  describe a central extension $\G_p(M)$ of the groupoid $\G(M)$ by
an abelian group  $\Omega^{2n-2}_\text{closed}(M)$. The set of objects is again $\mathcal{G}(M)$,
and the set of morphisms is labeled by triples $(g, A, \alpha)$, where $g \in G(M),
A \in \mathcal{G}(M)$ and $\alpha \in \Omega^{2n-2}(M)$ such that
$d\alpha=e_p(A,F)-e_p(A^g,F_{A^g})$. The composition of morphisms is given by 
$$
m_{g,A,\alpha}m_{h,A^g,\beta}=m_{gh,A,\alpha+\beta},
$$
where
$$
\begin{array}{lll}
d(\alpha+\beta) & = & e_p(A,F_A)-e_p(A^g,F_{A^g})+e_p(A^{g},F_{A^g})-e_p(A^{gh},F_{A^{gh}}) \\
& =& e_p(A,F_A)-e_p(A^{gh},F_{A^{gh}}).
\end{array}
$$

Note that for the
groupoid of global sections of $\G_p(M)$, the natural projection to $G(M)$ may no longer be
surjective. Indeed, the cohomology class of $e_p(A,F_A)-e_p(A^g,F_{A^g})$ in $H^{2n-1}(M)$ coincides
with $[g^* \eta_p] \in H^{2n-1}(M)$, where $\eta_p=e_p(\theta, 0) \in (\wedge \g^*)^\g \subset \Omega(G)$.
If $g^* \eta_p \neq 0$, $\G_p(M)$ does not contain elements which project to $g$.

A different choice of $e_p$ gives rise to an isomorphic sheaf of 
groupoids with an isomorphism given by $\alpha \mapsto \alpha'=\alpha+f(A,F)-f(A^g,F_{A^g})$.
In the Physics literature, $\alpha$ is called the Wess-Zumino action.

Recall that the DGLA $D_p\g$  is a central extension of $D\g$ by the line $\R[2n-2]$. 
Note that $\SA(\R[2n-2],M) \cong 
\Omega(M)^{2n-2}_\text{closed}$. Since $D\g$ is acyclic, we obtain an exact sequence 
of sheaves of Lie algebras
$$
0\to \Omega(M)^{2n-2}_{\rm closed} \to \SA(M,D_p\g) \to \SA(M,D\g) \to 0.
$$
This exact sequence integrates to an exact sequence of sheaves of groups,
$$
1 \to \Omega(M)^{2n-2}_{\rm closed} \to D_pG(M) \to DG(M).
$$

Consider a map $\mu_p: \G_p(M) \to \overline\SG(M,D_p\g,G)$ defined by formula
\begin{equation}  \label{mu2}
\mu_p(m_{g,A,\alpha}) = 
\Phi_p(A, F_A) g\, \Phi_p(A^g,F_{A^g})^{-1} e^{\alpha \otimes c} ,
\end{equation}
where $c$ is the generator of the central line of $D_p\g$, and $\Phi_p$ is defined 
in Section 5.

\begin{prop} \label{prop7}
The map $\mu_p$ is a morphism of groupoids, it takes values 
in $D_pG(M) \subset \overline\SG(M,D_p\g,G)$, and it restricts to identity on
$\Omega^{2n-2}_\text{closed}$.
\end{prop}

\begin{proof}
The proof is similar to Theorem \ref{prop:mu1}. We have,
$$
d \mu_p(m_{g,A,\alpha})=\mu_p(m_{g,A,\alpha}) (e(A^g,F_{A^g})-e(A,F_A)+d\alpha)\otimes c =0.
$$
Hence, $\mu_p$ takes values in $D_pG(M)$.  
For the morphism of groupoids, one follows the proof of Proposition \ref{prop:mu1}
and uses the fact that $m_{g,A,\alpha} \circ m_{f,A^g,\beta}=m_{g,A,\alpha+\beta}$.
Finally, for $g=1$ we obtain $\mu_p(m_{1,A,\alpha})=\exp(\alpha \otimes c)$ which
coincides with the image of $\alpha$ under injection $\Omega^{2n-2}_{\rm closed} \to D_pG(M)$.
\end{proof}

\subsubsection{Structure of the group $D_pG(M)$}
The DGLA ${\rm ker}(D_p\g \to C\g)$ fits into a short exact sequence
$$
0 \to \R[2n-2] \to {\rm ker}(D_p\g \to C\g) \to {\rm ker}(D\g \to C\g) \to 0
$$
giving rise to a short exact sequence of current algebras
$$
0 \to \Omega^{2n-2}_{\rm closed}(M)  \to \SA(M, {\rm ker}(D_p\g \to C\g)) \to \SA(M, {\rm ker}(D\g \to C\g)) \to 0.
$$
All of these Lie algebras being nilpotent, the exact sequence lifts to an exact sequence of groups
\begin{equation} \label{exact_sequence}
1 \to \Omega^{2n-2}_{\rm closed} \to \SG(M, {\rm ker}(D_p\g \to C\g),1) \to \SG(M, {\rm ker}(D\g \to C\g),1) \to 1.
\end{equation}
Furthermore, by choosing a section of the projection $\SA(M, {\rm ker}(D_p\g \to C\g)) \to \SA(M, {\rm ker}(D\g \to C\g))$,
and by composing with the exponential map we obtain a section
$$
\nu_p:  \left(\bigoplus_{i>0}\Omega^i(M)\otimes U^{-i}\right) \cong \SA(M, {\rm ker}(D\g \to C\g)) \to \SG(M, {\rm ker}(D_p\g \to C\g),1).
$$

Recall that $\eta_p=e_p(\theta, 0) \in \Omega^{2n-1}(G)$.
\begin{prop}
The image of the natural projection $D_pG(M) \to DG(M)$ is the set of elements of $DG(M)$
which project to  maps $g: M \to G$ with vanishing $[g^* \eta_p] \in H^{2n-1}(M)$.
\end{prop}

\begin{proof}
Let $f$ be an element of $DG(M)$, and  $g$ be the projection of $f$ to $G(M)$.
Then, $f_0=\mu(m_{g,0})^{-1}f$ projects to the group unit of
$G(M)$. That is, $f_0 \in \SG(M, {\rm ker}(D\g \to C\g),1)$. The exact sequence
\eqref{exact_sequence} implies that $f_0$ admits a lift to
$\SG(M,{\rm ker}(D_p\g \to C\g),1) \subset D_pG(M)$. Hence, $f$ admits a lift to
$D_pG$ if and only if so does $\mu(m_{g,0})$.

Recall that, as a graded Lie algebra, $D_p\g$ is a direct sum 
$D_p\g =D\g\oplus\mathbb{R}c$ of $D\g$ and the central line $\R c$ with $c$
the generator of degree $2-2n$. Hence, we have 
$\overline\SG(M, D_p\g,G)= \overline\SG(M, D\g,G)\times\exp(\Omega^{2n-2}(M)\otimes c)$.
Let us consider the subgroup
$$
Q=DG(M)\times\exp\big(\Omega^{2n-2}(M)\otimes c\big) \subset \overline\SG(M, D_p\g,G)
$$
containing $D_pG(M)$. Lifts of $\mu(m_{g,0})$ to $Q$ are of the form
$\hat{g}=\mu(m_{g,0}) \exp(\alpha \otimes c)$, where $\alpha\in \Omega^{2n-2}(M)$.
Since $\mu(m_{g,0})^{-1}d_p\mu(m_{g,0})=g^*\eta_p\otimes c$, we have
$$
\hat{g}^{-1}d_p\hat{g}=(g^*\eta_p + d\alpha) \otimes c.
$$
If $[g^*\eta_p]\neq 0$, then $\hat{g}^{-1}d_p \hat{g}\neq 0$, and $\mu(m_{g,0})$ does
not lift to $D_pG(M)$. If $[g^*\eta_p]=0$, we can achieve $\hat{g}\in D_pG(M)$ for a suitable choice of $\alpha$.
\end{proof}

Let us denote by ${}_pG(M)$ the set of pairs
$(g:M\to G, \alpha\in\Omega^{2n-2}(M))$ such that $g^*\eta_p + d \alpha=0$. 
In other words, ${}_pG(M)$ is the set of morphisms in the groupoid $\hat{\mathcal{G}}_p(M)$ starting at the object $A=0$.

\begin{prop}
The map $(g,\alpha, u) \to \mu_p(m_{g,0,\alpha}) \nu_p(u)$ defines a bijection
\begin{equation}\label{eq:explDpg}
{}_pG(M)\times   \left(\bigoplus_{i>0}\Omega^i(M)\otimes U^{-i}\right) \to D_pG(M).
\end{equation}
The maps  \eqref{eq:explDg} and \eqref{eq:explDpg} and the natural projections ${}_pG(M)\to G(M)$ and $D_pG(M)\to DG(M)$ form a commutative diagram.
\end{prop}

\begin{proof}
For the first statement, let $f \in D_pG(M)$, and denote its image in $G(M)$ by $g$. Then, there is a form $\alpha \in \Omega^{2n-2}(M)$ such that
$g^* \eta_p + d\alpha =0$, and $\hat{g}:=\mu_p(m_{g,0,\alpha}) \in D_pG(M)$ with the same projection $g \in G(M)$.
Hence, $\hat{g}^{-1} f \in \SG(M, {\rm ker}(D_p\g \to C\g)$, and it is of the form $\exp(u')$ for some 
$u' \in  \SA(M, {\rm ker}(D_p\g \to C\g))$.
The element $u'$ projects to $u \in \left(\bigoplus_{i>0}\Omega^i(M)\otimes U^{-i}\right) \cong \SA(M, {\rm ker}(D\g \to C\g)$. Then, $\hat{g}^{-1}f\nu_p(u)^{-1}$ projects
to the group unit in $\SG(M, {\rm ker}(D_p\g \to C\g),1)$. Hence,  it is of the form $\exp(\beta \otimes c)$ for some $\beta \in \Omega^{2n-2}_{\rm closed}(M)$.
For the element $f$ we obtain $f=\mu_p(m_{g,0,\alpha}) \exp(\beta \otimes c) \nu_p(u) = \mu_p(m_{g,0,\alpha+\beta}) \nu_p(u)$, as required.

For the second statement, the natural projections  ${}_pG(M)\to G(M)$ and $D_pG(M)\to DG(M)$ are forgetful maps with respect to differential forms $\alpha \in \Omega^{2n-2}(M)$.
This implies commutativity of the diagram announced in the Proposition.
\end{proof}

\subsection{Torsors and obstructions}

A torsor over the sheaf of groupoids $\G(M)$ is a principal $G$-bundle over $M$ with a choice of a principal $\g$-connection. More explicitly, if $U_i$ is an open cover of $M$, a (descent data for a) $\G(M)$-torsor on $M$ is defined by a choice of local $\g$-connections $A_i\in \Omega^1(U_i) \otimes \g$ and of the gluing maps $g_{ij}:U_{ij}\to G$ such that $g_{ij}$ is a cocycle, and $A_j=A_i^{g_{ij}}$.

\begin{thm}
Let $\mathcal{T}$ be a $\G(M)$-torsor over $M$ with underlying principal $G$-bundle $P$.
It  lifts to a $\G_p(M)$-torsor if and only if the Chern-Weil class
${\rm cw}(p)=[p(F)]$ of $P$ vanishes.
\end{thm}

\begin{proof}
A lift of a $\G(M)$-torsor to a $\G_p(M)$-torsor amounts to a choice of a Cech cocycle $\alpha_{ij}\in \Omega^{2n-2}(U_{ij})$  such that $d\alpha_{ij}=e_p(A_i,F_{A_i})-e_p(A_j,F_{A_j})$.
Assmue that such a cocycle exists. 
Since the Cech cohomology  $H^1(\Omega^{2n-2}(M))$ vanishes, there exist (for a sufficiently fine cover) local forms $\beta_i \in \Omega^{2n-2}(U_i)$ such that $\alpha_{ij}=\beta_i-\beta_j$. Then, the local forms
$e_p(A_i,F_{A_i})-d\beta_i=e_p(A_j,F_{A_j}) - d\beta_j$ define a global section of $\Omega^{2n-1}(M)$.
The de Rham differential of this globally defined differential form is
$d(e_p(A_i,F_{A_i}) - d\beta_i)=de_p(A_i,F_{A_i})=p(F_{A_i})$. Hence, ${\rm cw}(p)=[p(F_{A_i})]=0$ in $H^{2n}(M)$.

In the other direction: if $[p(F_{A_i})]=0$, there is a differential form $\gamma \in \Omega^{2n-1}(M)$ such that $d\gamma=p(F_{A_i})$ on $U_i$. Let $\omega_i=e_p(A_i,F_{A_i})-\gamma \in \Omega^{2n-1}(U_i)$. We have $d\omega_i=0$, and if the cover is sufficiently fine
we find $\beta_i \in \Omega^{2n-2}(U_i)$ such that $d\beta_i=\omega_i$. Put
$\alpha_{ij}=\beta_i-\beta_j$. We have, 
$$
d\alpha_{ij}=\omega_i-\omega_j=(e_p(A_i,F_{A_i})-\gamma)-(e_p(A_j,F_{A_j})-\gamma)=
e_p(A_i,F_{A_i})-e_p(A_j,F_{A_j}),
$$
as required.
\end{proof}

The groupoid morphism $\mu: \G(M) \to DG(M)$ can be used to map $\G(M)$-torsors to $DG(M)$-torsors. Since its lift $\mu_p$ is equal to identity on $\Omega^{2n-2}_{\rm closed}$
(see Proposition \ref{prop7}), 
the obstruction to lifting the corresponding $DG(M)$-torsor to a $D_pG(M)$-torsor is again ${\rm cw}(p)$.

Finally, let us describe general $DG(M)$-torsors. It is enough to notice that $DG(M)$ is an extension of $G(M)$ by a sheaf of nilpotent groups and that this sheaf is acyclic (this follows from the acyclicity of the kernel of $D\g\to C\g$). Therefore, the classification of $DG(M)$-torsors is the same as the classification of principal $G$-bundles ($G(M)$-torsors).
We thus have the following result.

\begin{thm}   \label{torsor}
  The classification of $DG(M)$-torsors is the same as the classification of principal $G$-bundles; the correspondence is given by the morphism $DG(M)\to G(M)$. The obstruction to lifting a $DG(M)$-torsor to a $D_pG(M)$-torsor is exactly the Chern-Weil class ${\rm cw}(p)$.
\end{thm}

\section{Groups integrating current algebras $\CA(M, A)$}
\newcommand{\CG}{\mathcal{CG}}
The purpose of this Section is to construct groups integrating current algebras $\CA(M, A)$. The construction is similar to the integration methods in \cite{LMNS} and \cite{Vizman}.

\subsection{Integration of $\CA(M, A)$}
We will need the following notation: for an embedding of manifolds $f: Y \to X$, we denote
$\Omega(X, Y):= {\rm ker}(f^*: \Omega(X) \to \Omega(Y))$ (this complex is quasi-isomorphic
to the standard relative de Rham complex). Note that $\SA(X,Y,A)=(\Omega(X,Y) \otimes A)^0_{\rm closed}$
is a Lie subalgebra of $\SA(X, A)$, and $(\Omega(X,Y) \otimes A)^0_{\rm exact} \subset \SA(X,A)$
is a Lie ideal.

\begin{prop}
Let $I=[0,1]$ be the unit interval with coordinate $s$. The map $\tau: \alpha \mapsto d(s\alpha)$ induces a Lie algebra isomorphism
\begin{equation}\label{eq:paths}
\mathcal{CA}(M,A)\cong\frac{\SA(M\times I,M\times\{0\}, A)}
{\Bigl(\Omega\bigl(M\times I,M\times\{0,1\}\bigr)\otimes A\Bigr)^0_{\rm exact}}.
\end{equation}
\end{prop}

\begin{proof}
The map $\tau$ is well-defined since $\tau(d\omega)=d(s d\omega) =- d(ds \wedge \omega)$, and 
$ds \wedge \omega$ vanishes when restricted to $M\times \{ 0\}$ and $M\times \{ 1 \}$.
To show that $\tau$ is an isomorphism of vector spaces, observe that the fiber integral
$\lambda \mapsto \int_I \lambda$ is an inverse of $\tau$. Finally, $\tau$ is a Lie homomorphism
since for $\alpha, \beta \in \bigl(\Omega(M)\otimes A\bigr)^{-1}$ we have 
$$
\tau([\alpha,d\beta])=d(s[\alpha,d\beta])=[d(s\alpha),d(s\beta)]+d[s\alpha,d((1-s)\beta)]\equiv [d(s\alpha),d(s\beta)].
$$
Here we used that $[s\alpha,d((1-s)\beta)]$ vanishes on $M\times \{ 0\}$ and $M\times \{ 1 \}$.
\end{proof}

We can reformulate the isomorphism \eqref{eq:paths} as follows. Let a \emph{path} in $\SA(M,A)$ be an element $\gamma\in\SA(M\times I, A)$. The endpoints of $\gamma$ are the elements $\gamma|_{M\times\{0\}}, \gamma|_{M\times\{1\}} \in \SA(M,A)$. Let  $\gamma_0,\gamma_1\in\SA(M\times I, A)$ be two paths with the same endpoints $\epsilon_0$ and $\epsilon_1$, i.e.\ such that 
$$\gamma_0|_{M\times\{0\}}=\gamma_1|_{M\times\{0\}}=\epsilon_0\text{ and } \gamma_0|_{M\times\{1\}}=\gamma_1|_{M\times\{1\}}=\epsilon_1.$$ 
A \emph{homotopy} between $\gamma_0$ and $\gamma_1$ is an element $\chi\in\SA(M\times I\times I)$ such that $\gamma_0=\chi|_{M\times I\times\{0\}}$, $\gamma_1=\chi|_{M\times I\times\{1\}}$, and  $\chi|_{M\times\{0\}\times I}$ is the pullback of $\epsilon_0$ and $\chi|_{M\times\{1\}\times I}$  is the pullback of $\epsilon_1$ under  projection $M\times I\to M$.

Equation \eqref{eq:paths} says that $\CA(M,A)$ is isomorphic to the Lie algebra of paths in $\SA(M,A)$ starting at $0\in\SA(M,A)$, modulo homotopy of paths.
Indeed, for  $\gamma_1=\gamma_0+d\mu$ the desired homotopy is $\chi=\gamma_0+d(t\mu)$ (here $t$ is the parameter on the unit segment).
In the other direction, if $\chi$ is a homotopy interpolating between $\gamma_0$ and $\gamma_1$, we have
$$
\gamma_1-\gamma_0=d\int_t\chi\in  \Bigl(\Omega\bigl(M\times I,M\times\{0,1\}\bigr)\otimes A\Bigr)^0_{\rm exact}
$$
as required.

In the same way, we can introduce paths and their homotopies in $\SG(M,A,G)$. The group $\CG(M,A,G)$ is then defined as the group of paths in $\SG(M,A)$ 
starting at the group unit, modulo homotopy of paths. Again, it depends on the choice of 
a connected Lie group $G$ integrating the Lie algebra $\g=A^0_{\rm closed}$.

\begin{rem}
Let $(A,G) \to (B,H)$ be a morphism of pairs consisting of a DGLA homomorphism
$A\to B$ and a group homomorphism $G\to H$ integrating the Lie homomorphism
$\g=A^0_{\rm closed} \to B^0_{\rm closed} =\mathfrak{h}$. Then, similarly to the $\SG$ functor, we obtain a group homomorphism $\CG(M,A,G) \to \CG(M,B,G)$.
For example, consider the DGLA homomorphism $D_p\g \to A_{FMS}$.
Note that $({A_{FMS}})^0_\text{closed} \cong \g \ltimes W\g^{2n-2}_\text{closed}$. If $G$ is a connected
Lie group integrating $\g$, the Lie algebra $A^0_\text{FMS}$ integrates to
the semi-direct product $H=G\ltimes W\g^{2n-2}_\text{closed}$. Obviously, we have a
morphism of pairs $(D_p\g, G) \to (A_{FMS}, H)$. It induces
a group homomorphism 
$\CG(M,D_p\g,G) \to \CG(M, A_{FMS}, H)$ from $\CG(M, D_p\g, G)$ to the group  
$CG(M, A_{FMS}, H)$ integrating 
Faddeev-Mickelsson-Shatashvili (FMS) current algebra (for details about the group integrating the FMS current algebra see the book \cite{mibook}). The image of this map is
the group $\CG(M, B_{FMS}, G)$ integrating the truncated FMS current algebra
(here $B_{FMS}={\rm im}(D_p\g \to A_{FMS})$ and ${\rm im}(G \to H)=G$).
\end{rem}

\subsection{The case of $A=D_p\g$ and $M$ a sphere}
In this Section we restrict our attention to examples of $A=C\g, D\g, D_p\g$  
and $M=S^{n}$ a sphere. Let us introduce the shorthand notation $\widetilde G(M)=\CG(M,C\g,G)$, $\widetilde{DG}(M)=\CG(M,D\g,G)$, and $\widetilde{D_pG}(M)=\CG(M,D_p\g,G)$.

\begin{prop} \label{pi}
Let $G$ be a simply connected Lie group.
Then, there is an  exact sequence of groups,
$$1\to\pi_{n+1}(G)\to\widetilde G(S^n)\to G(S^n)\to\pi_n(G)\to 1.$$
\end{prop}

\begin{proof}
The group $\widetilde G(S^n)$ consists of paths $g_t$  in the group $G(S^n)$ which start at the group unit ($g_0=1$) modulo homotopy. 
Note that $\pi_0(G(S^n)) \cong \pi_n(G)$ and,  if $G$ is simply connected, $\pi_1(G(S^n)) \cong \pi_{n+1}(G)$.
This implies the exact sequence in the Proposition.
\end{proof}

\begin{prop}
Let $G$ be a simply connected Lie group.
Then, there is an  exact sequence of groups,
$$1\to\pi_{n+1}(G)\to\widetilde{DG}(S^n)\to DG(S^n)\to\pi_n(G)\to 1.$$
\end{prop}

\begin{proof}
The group $\widetilde {DG}(S^n)$ can be described using the bijection \eqref{eq:explDg}. Let us define paths and their homotopies in  
$\bigoplus_{i>0}\Omega^i(M)\otimes U^{-i}$ in the same way as above. In this sense, $\bigoplus_{i>0}\Omega^i(M)\otimes U^{-i}$ is 1-connected.
We infer from \eqref{eq:explDg} that there is a bijection
$$\widetilde{DG}(M)\to \widetilde G(M)\times \left( \bigoplus_{i>0}\Omega^i(M)\otimes U^{-i}\right) .$$
Together with the exact sequence of the Proposition \ref{pi}, it implies the required exact sequence.
\end{proof}

For $p\in (S^n\g^*)^\g$, let $\eta_p =e_p(\theta,0)\in \Omega^{2n-1}(G)$ be the bi-invariant differential form on $G$
defined by transgression, and $\Pi: \pi_{2n-1}(G) \to \mathbb{R}$  be the group homomorphism defined
by the integration map: $C \mapsto \int_C \eta_p$. The image of $\Pi$ is a subgroup of $\mathbb{R}$.
We will be interested in the quotient $\mathbb{R}/{\rm im}(\Pi)$.
If $p \in \left( (S^+\g)^\g \right)^2$, then $\eta_p=0$ and the quotient is equal to $\mathbb{R}$.
If $\g$ is a Lie algebra of a compact simple Lie group, and 
$p$ is a generator of $(S\g)^\g$ of degree $m_i+1$ (here $m_i$
is one of the exponents of $\g$), and the multiplicity of $m_i$ is equal to one (this is always the case with 
the exception of one of the exponents of the group $SO(2n)$), then $\mathbb{R}/{\rm im}(\Pi) \cong S^1$.

\begin{thm} \label{thm:extensionS^1}
Let $p\in (S^n\g^*)^\g$. Then, there is an exact sequence of groups
$$
1 \to \mathbb{R}/{\rm im}(\Pi) \to \widetilde{D_pG}(S^{2n-3})\to\widetilde{DG}(S^{2n-3}) \to 1,
$$
where $\mathbb{R}/{\rm im}(\Pi)$ is a central subgroup.
\end{thm}

\begin{proof}
The bijection \eqref{eq:explDpg} implies the bijection
$$
\widetilde{D_pG}(M)\to \widetilde{{}_pG}(M)\times \left( \bigoplus_{i>0}\Omega^i(M)\otimes U^{-i}\right) ,
$$
where $\widetilde{{}_pG}(M)$ stands (as usual) for the set of  paths in ${}_pG(M)$ starting at the group unit
modulo homotopy. 

For $M=S^{2n-3}$,  $\Omega^{2n-2}(S^{2n-3})=0$ and therefore  ${}_pG(S^{2n-3})=G(S^{2n-3})$.
A path in ${}_pG(S^{2n-3})$ is a pair $(g_t,\alpha)$ where $g_t:S^{2n-3}\times I\to G$ and $\alpha\in\Omega^{2n-2}(S^{2n-3}\times I)$. 
Since $\alpha$ is a top degree form, there are no conditions imposed on it and
the group homomorphism $\widetilde{D_pG}(S^{2n-3})\to\widetilde{DG}(S^{2n-3})$ is surjective.

 %A homotopy of paths in ${}_pG(S^{2n-3})$ is a homotopy $h:S^{2n-3}\times I\times I\to G$ together with a differential 
 %form $\beta\in\Omega^{2n-2}(S^{2n-3}\times I\times I)$ such that $d\beta=h^*\eta_p$. 
Let us determine the kernel of the group homomorphism
$\widetilde{D_pG}(S^{2n-3})\to\widetilde{DG}(S^{2n-3})$, or equivalently, the kernel of the map $\widetilde{{}_pG}(S^{2n-3})\to\widetilde{G}(S^{2n-3})$. It consists of homotopy classes of paths in ${}_pG(S^{2n-3})$ of the form
$(1,\alpha)$, where $\alpha\in\Omega^{2n-2}(S^{2n-3}\times I)$  and $1$ denotes the constant map to $1\in G$. A homotopy between two such paths $(1,\alpha_0)$, $(1,\alpha_1)$ is a pair $(h,\beta)$, where $h$ is a map $h:S^{2n-3}\times I\times I\to G$ such that $h|_{S^{2n-3}\times\partial (I\times I)}=1$ and $\beta\in\Omega^{2n-2}(S^{2n-3}\times I\times I)$ is such that $d\beta+h^*\eta_p=0$, and $\beta|_{S^{2n-3}\times\{0,1\}\times I}=0$, $\beta|_{S^{2n-3}\times I\times\{0\}}=\alpha_0$, $\beta|_{S^{2n-3}\times I\times\{1\}}=\alpha_1$. The Stokes theorem implies
$$
\int_{S^{2n-3}\times I}(\alpha_1-\alpha_0)=- \int_{S^{2n-3}\times I\times I}h^*\eta_p\in\text{im}(\Pi).
$$

I the other direction, let $(1,\alpha_0)$ and $(1,\alpha_1)$ be paths in ${}_pG(S^{2n-3})$ such that $\int_{S^{2n-3}\times I}(\alpha_1-\alpha_0)=\Pi(a)$ for some $a\in\pi_{2n-1}(G)$. Choose a smooth map $h:S^{2n-3}\times I\times I\to G$, $h|_{S^{2n-3}\times\partial (I\times I)}=1$ representing the class $(-a)$. Then, there exists a differential form
$\beta\in\Omega^{2n-2}(S^{2n-3}\times I\times I)$ such that $(h,\beta)$ is a homotopy between the paths $(1,\alpha_0)$ and $(1,\alpha_1)$.

As a result, two paths $(1,\alpha_0)$ and $(1,\alpha_1)$ are homotopic if and only if 
$$\int_{S^{2n-3}\times I}(\alpha_1-\alpha_0)\in\text{im}(\Pi).$$ The kernel of 
the map $\widetilde{D_pG}(S^{2n-3})\to\widetilde{DG}(S^{2n-3})$
is therefore isomorphic to $\mathbb{R}/\text{im}(\Pi)$.

\end{proof}

\subsection{Example: central extensions of loop groups}

Let $\g$ be a Lie algebra of a simple simply connected compact Lie group $G$,
and let $p\in (S^2\g)^\g$ be a nonvanishing element (note that $(S^2\g)^\g \cong \R)$).
In this case, $\R/{\rm im}(\Pi) \cong S^1$. 

Taking $M=S^1$ in the discussion of the previous section, we observe the following. 
Since $\pi_1(G)=\pi_2(G)=0$, we have $\tilde{G}(S^1)\cong G(S^1)$. In this very special case,
we have $\CA(S^1, D\g)=\CA(S^1, C\g)$, and $\widetilde{DG}(S^1)\cong G(S^1)$. 
Finally, for $\widetilde{D_pG}(S^1)$ we obtain an exact sequence of groups
$$
1 \to S^1 \to \widetilde{D_pG}(S^1) \to G(S^1) \to 1.
$$ 

\begin{prop}
The group $\widetilde{D_pG}(S^1)$ is the standard central extension of the 
loop group $LG=G(S^1)$.
\end{prop}

\begin{proof}
Let us recall the construction of the standard central extension $\widetilde{LG}$ of the loop group $LG$. Normalize $p\in (S^2\g^*)^\g$ (which is unique up to multiple) by the condition $\text{im}(\Pi)=\mathbb{Z}$, and by requiring $p$ to be positive-definite. Denote by 
$D^2$ the unit 2-dimensional disc. First, one introduces a central extension $\widehat{G}(D^2)$ of the group $G(D^2)$ by $U(1)$. As a set, $\widehat{G}(D^2)$ 
is the direct product $G(D^2)\times U(1)$, and the group law is given by formula
$$
(g_1,u_1)(g_2,u_2)=\Bigl(g_1g_2,u_1u_2\exp\bigl(\pi i\int_{D^2}
p\bigl(g_1^{-1}dg_1,dg_2g_2^{-1}\bigr)\Bigr).
$$
Note that the integrand is (up to a factor of $2\pi i$)  the pull back of the 2-form $\rho_p$ under
the map $(g_1,g_2): D^2 \to G\times G$.

Next, one introduces an equivalence relation on $\widehat{G}(D^2)$: $(g_1,u_1)\sim (g_2,u_2)$ if there exists a homotopy $h:D^2\times I\to G$ between $g_1$ and $g_2$ relative to the boundary $S^1$ of $D^2$, such that $u_2=u_1\exp\int_{D^2\times I} h^*\eta_p$. The group $\widetilde{LG}$ is then defined as the quotient of the group $\widehat{G}(D^2)$ by the equivalence relation $\sim$.

Equivalently, $\widetilde{LG}$ is quotient of the group $C_pG(D^2)=\SG(D^2,C_p\g,G)$ by the following equivalence relation: $(g_1,\omega_1)$ and $(g_2,\omega_2)$ are equivalent if there exists a homotopy $(h,\chi)\in C_pG(D^2\times I)$ between $(g_1,\omega_1)$ and $(g_2,\omega_2)$ relative to $\partial D^2$. Indeed, the map $C_pG(D^2)\to \widehat{G}(D^2)G$
given by  $(g,\alpha)\mapsto(g,\exp2\pi i\int_{D^2}\alpha)$ is a surjective group homomorphism, and two elements of $C_pG(D^2)$ are equivalent if and only if their images are equivalent.

Let us introduce a modification $\widehat{G}'(D^2)$ of $\widehat{G}(D^2)$: in its definition we replace the group $G(D^2)$ by the subgroup of $G(S^1\times I)$ of maps $g:S^1\times I\to G$ such that $g|_{S^1\times\{0\}}=1$, i.e.\ by the group of paths in $G(S^1)=LG$ starting at $1\in G(S^1)$. We also replace the homotopies $h$ by homotopies of paths. Then, it is easy to see that $\widehat{G}'(D^2)$ modulo the equivalence is again equal to $\widetilde{LG}$. Equivalently, it is the group of paths in $C_pG(S^1)$ starting at 1 modulo homotopy, i.e.\ the group $\CG(S^1, C_p\g, G)$. Hence, $\widetilde{LG}\cong \CG(S^1, C_p\g, G)$.
\end{proof}


\begin{thebibliography}{99}

\bibitem{AM}
A. Alekseev, E. Meinrenken,
The non-commutative Weil algebra,  Invent. Math. {\bf 139} (2000), no. 1, 135--172.

\bibitem{Cartan}
H. Cartan,
Notions d'alg\`ebre diff\'erentielle; application aux groupes de Lie et aux vari\'et\'es o\`u op\`ere un groupe de Lie (French).
Colloque de topologie (espaces fibr\'es), Bruxelles, 1950, pp. 15--27.

H. Cartan,
La transgression dans un groupe de Lie et dans un espace fibr\'e principal (French),
Colloque de topologie (espaces fibr\'es), Bruxelles, 1950, pp. 57--71.

\bibitem{dr}
V. Drinfeld, Quasi-Hopf algebras,  Algebra i Analiz  1  (1989),  no. 6, 114--148.

\bibitem{F}
L. Faddeev, Operator anomaly for the Gauss law,  Phys. Lett. {\bf B 145} (1984), no. 1-2, 81--84.

\bibitem{FS}
L. Faddeev, S. Shatashvili, Algebraic and Hamiltonian methods in the theory of nonabelian anomalies (Russian),
Teoret. Mat. Fiz. {\bf 60} (1984), no. 2, 206--217.

\bibitem{GK} V. Ginzburg, M. Kapranov, Koszul duality for operads. Duke Math. J. {\bf 76} (1994), no. 1, 203-272.

\bibitem{GS}
V. Guillemin, S. Sternberg, Supersymmetry and equivariant de Rham theory. Mathematics Past and Present. Springer-Verlag, Berlin, 1999.

\bibitem{HU}
S. Hu, B. Uribe, Extended manifolds and extended equivariant cohomology, J. Geom. Phys. {\bf 59} (2009), no. 1, 104--131.

\bibitem{yks}
Y. Kosmann-Schwarzbach, From Poisson algebras to Gerstenhaber algebras, Ann. Inst.
Fourier (Grenoble) 46 (1996), 1241--1272.

\bibitem{LMNS}
A. Losev, G. Moore, N. Nekrasov, S. Shatashvili, Central extensions of gauge groups revisited, Selecta Mathematica {\bf 4} No.1 (1998), 117-123.

\bibitem{Mickelsson}
J. Mickelsson, Chiral anomalies in even and odd dimensions, Commun. Math. Phys. {\bf 97} (1985) 361--370.

\bibitem{mibook}
J. Mickelsson, Current algebras and groups, Plenum Press, New York, 1989.

\bibitem{Vizman}
C. Vizman, The path group construction of Lie group extensions, Journal of Geometry and Physics {\bf 58} No.7 (2008), 860-873.

\end{thebibliography}
\end{document}